\documentclass[%
a4paper]{mpi2015-cscpreprint}

\usepackage{graphicx,epstopdf}
\usepackage{subfig}
\usepackage{enumitem} % enumerate
\usepackage{tikz}

\usepackage{breqn} % breaks the equation
\usepackage{algorithm}
\usepackage{algorithmic}
\usepackage{lipsum}
\usepackage{amsfonts}

\newtheorem{remark}{Remark}

\usepackage{amssymb}
\usepackage{bbm}

   \newtheorem{theorem}{Theorem}

      \newtheorem{proof}{Proof}

\newcommand{\td}{\tilde}
\newcommand{\mbb}{\mathbb}
\newcommand{\mcal}{\mathcal}
\newcommand{\parSp}{\mcal Z}		% parameter domain
\newcommand{\opn}{\operatorname}
\newcommand{\hsp}{\hspace{0.1cm}}
\newcommand{\hspB}{\hspace{0.3cm}}
\newcommand{\pd}{\partial}
\newcommand{\lan}{\left\langle}
\newcommand{\ran}{\right\rangle}

\newcommand{\ttrain}[1]{t_{#1}}

\newcommand{\tref}{t_{\opn{ref}}}

\newcommand{\snapMat}[3]{\mcal S_{#2,#3}({#1})}

\DeclareMathOperator*{\argmin}{arg\,min}

\newcommand{\sol}[1]{u(\cdot,{#1})}
\newcommand{\solFV}[1]{u_N(\cdot,{#1})}

\newcommand{\solROM}[2]{u_{#1}(\cdot,#2)}

\newcommand{\dev}[2]{\Psi(\cdot,#2)}

\usetikzlibrary{decorations.pathreplacing}
\usetikzlibrary{decorations.shapes}
\usetikzlibrary{calc}

\tikzset{decorate sep/.style 2 args=
{decorate,decoration={shape backgrounds,shape=circle,shape size=#1,shape sep=#2}}}

\usepackage{amsopn}

\begin{document}
\title{Learning Reduced Order Models from Data for Hyperbolic PDEs}
  
\author[$\ast$]{Neeraj Sarna}
\affil[$\ast$]{Max Planck Institute for Dynamics of Complex Technical Systems, 39106 Magdeburg, Germany.
  \email{sarna@mpi-magdeburg.mpg.de}, \orcid{0000-0003-0607-2067}}
  
\author[$\dagger$]{Peter Benner}
\affil[$\dagger$]{Max Planck Institute for Dynamics of Complex Technical Systems, 39106 Magdeburg, Germany, and Faculty for Mathematics, Otto von Guericke University Magdeburg, 39106 Magdeburg, Germany.\authorcr
  \email{benner@mpi-magdeburg.mpg.de}, \orcid{0000-0003-3362-4103}}
  
\shorttitle{Learing Reduced order models for Hyperbolic PDEs}
\shortauthor{Neeraj Sarna and Peter Benner}
\shortdate{}
  


\abstract{Given a set of solution snapshots of a hyperbolic PDE, we are interested in learning a reduced order model (ROM). To this end, we propose a novel decompose then learn approach. We decompose the solution by expressing it as a composition of a transformed solution and a de-transformer. Our idea is to learn a ROM for both these objects, which, unlike the solution, are well approximable in a linear reduced space. A ROM for the (untransformed) solution is then recovered via a recomposition. The transformed solution results from composing the solution with a spatial transform that aligns the spatial discontinuities. Furthermore, the de-transformer is the inverse of the spatial transform and lets us recover a ROM for the solution. We consider an image registration technique to compute the spatial transform, and to learn a ROM, we resort to the dynamic mode decomposition (DMD) methodology. Several benchmark problems demonstrate the effectiveness our method in representing the data and as a predictive tool.}

\novelty{
\begin{enumerate}
\item A decompose then learn approach for hyperbolic PDEs.
\item The solution is decomposed into a transformed solution and a de-transformer.
\item A ROM is learned for the transformed solution and the de-transformer.
\item A re-composition recovers a ROM for the (untransformed) solution.
\item Numerical experiments showcase improvements over a standard DMD approach.
\end{enumerate}
%We perform model order reduction for problems with parameter-dependent jump-discontinuities by combining image registration with regression. Image registration lets us compute a spatial transform, which, when composed with the solution, induces good approximability in a linear reduced space. We then use regression to approximate the map between the parameter domain and the expansion coefficients of the transformation solution. We also develop a de-transformation step where we approximate the inverse of the spatial transform using regression. With the de-transformation step, we finally recover an approximation for the (untransformed) solution. Numerical experiments verify that our technique offers improvement over standard methods. Furthermore, for moderate parameter domain dimensions, we report speed-ups to one to two orders-of-magnitude. 
}
\maketitle

\section{Introduction}
Scientific computing often involves applications where a numerical solution is sought at several different time or parameter instances. For such applications, a high-fidelity (HF) finite-element (or volume or difference) type solver can be prohibitively expensive and one seeks to replace it with an accurate and efficient surrogate. The reduced order modelling paradigm offers a reliable architecture to construct such surrogates. It performs the onerous task of snapshot collection offline, thus allowing for online efficiency \cite{RBBook}. We are interested in developing ROMs that can well approximate solutions to time-dependent, possibly non-linear, hyperbolic PDEs. Such PDEs are ubiquitous in science and engineering applications involving wave-type phenomenon, with the Euler and the wave equation being two classic examples.

Standard reduced order modelling techniques approximate the solution in a linear reduced space, which is usually a span of reduced or POD type basis vectors \cite{JanStamm,PeterReview}. Although a linear reduced approximation is accurate for a broad range of applications, it is largely inadequate for hyperbolic PDEs. The primary reason being moving spatial discontinuities or steep gradients. Even for smooth initial and boundary data, solutions to such PDEs exhibit spatial discontinuities that often move in space and thus trigger temporal discontinuities. Temporal (or parametric) discontinuities (or steep gradients) have mainly two negative consequences for the approximability in a linear reduced space. Firstly, a large reduced space might be required to well approximate the solution, hindering the efficiency of a ROM \cite{PeterBook,Welper2017,WaveKolmogorov}. Secondly, temporal discontinuities trigger oscillations in reduced basis vectors, resulting in an oscillatory approximation \cite{GNAT,Nair,Welper2017,Fresca2020}. These inadequacies have motivated several authors to seek non-linear approximations. Broadly speaking, these approximations rely either on: (i) Wasserstein spaces \cite{Metric2019,RimOT}; (ii) online adaptivity of basis \cite{Benjamin2018model,NonLinearMOR}; (iii) auto-encoders \cite{KevinAuto,kim2020,Fresca2020}; (iv) domain decomposition \cite{RBHyp,Constantine}, or (v) manifold calibration/transformation \cite{Cagniart2019,MojganiLagrangian,Nair,Welper2017,RegisterMOR,Taddei2020ST}.

Running parallel to these previous works, we propose a novel ROM technique for hyperbolic PDEs. Our interest lies in a non-intrusive data-driven approach where one learns a ROM from a set of solution snapshots, without having access to the underlying discrete evolution operators---in the context of problems that do not exhibit transport dominated effects, works along these lines can be found in \cite{JanNN2018,Guo2019,Qian2020,Kramer2020,NairNonIntrusive2013,AerofoilNN2020}. Such a configuration is appealing for applications where one relies either on experiments, or huge legacy/commercial codes, for snapshot collection. Note that intrusive techniques compute a ROM by projecting a PDE onto the reduced space, which requires access to the underlying discrete evolution operator---related to our context, a few of the many noteworthy works involving the intrusive technique can be found in \cite{Rowley2004,GNAT,RegisterMOR,MATS,RB_FV,Taddei2020ST}. 

\subsection{Decompose then learn approach}
Since solutions to hyperbolic PDEs have poor approximability in a linear reduced space, following the standard approach of learning a ROM directly for the solution can be ineffective. Therefore, we propose to express the solution $u(\cdot,t)$ as a composition of two functions, which, unlike the solution, are sufficiently regular along the temporal domain and thus (with some additional assumptions \cite{WelperAdaptive,SarnaCalib2020}) exhibit good approximability in a linear reduced space. Our idea is to learn a ROM for these two functions, followed by a re-composition that approximates the solution. 

We decompose the solution as 
\begin{gather}
u(\cdot,t) = g(\td\varphi(\cdot,t),t), \label{decompose u}
\end{gather}
where 
\begin{gather}
g(\cdot,t) := u(\varphi(\cdot,t),t),\hspB\text{and}\hspB \td\varphi(\cdot,t) := \varphi(\cdot,t)^{-1}.
\end{gather}
We refer to $\varphi(\cdot,t):\Omega\to\Omega$ as the spatial transform. When composed with the solution $\sol{t}$, it results in a transformed solution $g(\cdot,t)$, which, at least ideally, has no temporal discontinuities \cite{Nair,Welper2017,RegisterMOR}.  We refer to $\td\varphi$ as the \textit{de-transformer} since it reverts solution transformation. Akin to the transformed solution, the de-transformer is also sufficiently regular in time; as noted in \cite{SarnaCalib2020}, this claim requires some assumptions that we brief later. Owing to the temporal regularity of the transformed solution and the de-transformer, we can expect them to be well approximable in a linear reduced space and thus, an accurate ROM can be learned for these two objects. A ROM for the solution is then recovered via the above relation. \Cref{fig: learn ROM} provides a high-level summary of our idea.

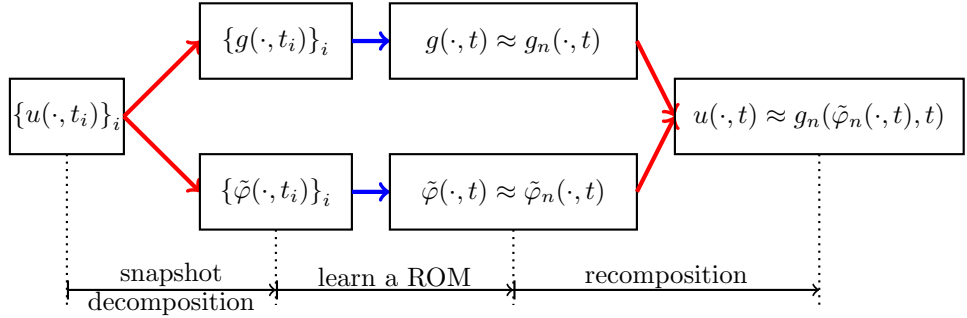
\begin{figure}[ht!]
\centering
\begin{tikzpicture}
% box containing solution snapshots
\draw [draw=black,thick] (-2,1) rectangle ++(1.5,1);
\node[] at (-2 + 0.75,1 + 0.5) {$\left\{u(\cdot,t_i)\right\}_i$};
% transformation step
\draw[->,ultra thick,red] (-0.5,1.5)--(0.5,2.5);
\draw[->,ultra thick,red] (-0.5,1.5)--(0.5,0.5);
% snapshots of the transformed solution
\draw [draw=black,thick] (0.5,2) rectangle ++(2,1);
\node[] at (0.5 + 1,2 + 0.5) {$\left\{g(\cdot,t_i)\right\}_i$};
% snapshots of the inverse of the spatial transform 
\draw [draw=black,thick] (0.5,0) rectangle ++(2,1);
\node[] at (0.5 + 1,0 + 0.5) {$\left\{\td\varphi(\cdot,t_i)\right\}_i$};
% learn model for transformed snapshots
\draw[->,ultra thick,blue] (2.5,2.5)--(3,2.5);
% learn model for the de-transformer
\draw[->,ultra thick,blue] (2.5,0.5)--(3,0.5);
% learned ROM for g
\draw [draw=black,thick] (3,2) rectangle ++(3.25,1);
\node[] at (3 + 1.625,2 + 0.5) {$g(\cdot,t)\approx g_n(\cdot,t)$};
% learned de-transformer
\draw[draw=black,thick] (3,0) rectangle ++(3.25,1);
\node[] at (3 + 1.625,0 + 0.5) {$\td\varphi(\cdot,t)\approx \td\varphi_n(\cdot,t)$};
% recombination phase
\draw[->,ultra thick,red] (6.25,0.5)--(6.75,1.5);
\draw[->,ultra thick,red] (6.25,2.5)--(6.75,1.5);
% recombine solution
\draw[draw=black,thick] (6.75,1) rectangle ++(3.8,1);
\node[] at (6.75 + 1.9,1 + 0.5) {$u(\cdot,t)\approx g_n(\td\varphi_n(\cdot,t),t)$};
% snapshot transformation phase
\draw[dotted,thick] (-1.25,1)--(-1.25,-1);
\draw [draw=black,thick,dotted] (1.5,0)--(1.5,-1);
\draw[|->,thick] (-1.25,-0.8)--(1.5,-0.8);
\node[] at (0.125,-0.60) {snapshot};
\node[] at (0.125,-0.98) {decomposition};
% learning ROM phase
\draw [draw=black,thick,dotted] (4.625,0)--(4.625,-1);
\draw[|->,thick] (1.5,-0.8)--(4.625,-0.8);
\node[] at (3.0625,-0.62) {learn a ROM};
% recomposition
\draw [draw=black,thick,dotted] (8.65,1)--(8.65,-1);
\draw[|->,thick] (4.625,-0.8)--(8.65,-0.8);
\node[] at (6.6375,-0.62) {recomposition};
\end{tikzpicture}
\caption{A high-level depiction of our idea. Snapshots are collected at the time-instances $\{t_i\}_{i=1,\dots,K}$. The standard approach directly learns a ROM for $\sol{t}$, without any prior decomposition. The functions $g_n$ and $\td\varphi_n$ represent ROMs for $g$ and $\td \varphi$, respectively. \label{fig: learn ROM}}	
\end{figure}

We find the following attributes desirable in a technique used to compute $\varphi$.
\begin{enumerate}
\item[(P1)] It should not require the discrete HF evolution operators.
\item[(P2)] It should not require an explicit form of the underlying hyperbolic PDE.
\item[(P3)] It should be independent of the technique used to learn the ROM.
\end{enumerate}
Violation of (P1) defeats the purpose to developing a non-intrusive technique. In a data-driven setting, an explicit form of the underlying PDE might be unavailable, thus (P2) is desirable. Note that the framework in \cite{Lu2020} does not satisfy (P2)---towards the end of this introduction, we elaborate further. 

Property (P3) has two crucial consequences. Firstly, it ensures that any of the different learning techniques---for instance, the operator inference approach \cite{Benjamin2016}, radial basis functions approach \cite{RBFXiao2015}, DMD approach \cite{DMDTheory}---can be used as a black-box to learn a ROM for $g$ and $\td\varphi$. Secondly, it detaches the decomposition and the re-composition step from any pre-existing numerical implementation of a ROM learning toolbox. Consequently, one does not need to maintain an additional implementation for PDEs that are well approximable via standard techniques. By simply choosing $\varphi(\cdot,t) = \opn{Id}$, where $\opn{Id}$ is the identity operator, and keeping the rest of the algorithm as such, one can recover a standard data-driven reduced order modelling technique. We resort to an optimization-based image registration technique to compute $\varphi$. Later sections make it obvious that this technique satisfies (P1)--(P3). 

We undertake the (plain) DMD approach to learn a ROM \cite{DMDKutz,DMDTheory,NonLinearFlows2009}. The DMD approach approximates the snapshot trajectory via the action of a best-fit linear operator. The resulting linear evolution equation is then projected onto the POD modes. Owing to the linearity, the reduced approximation is available in an explicit form and requires no further time stepping. This results in a highly efficient ROM, which, even for non-linear problems, is often orders of magnitude faster than a HF solver \cite{DMDKutz,MorDMD2014,RafiqDMD}. It is noteworthy that for non-linear problems, intrusive techniques are usually as expensive as HF solvers and require an additional layer of hyper-reduction to restore efficiency \cite{DEIMNonLinear}.

\subsection{Relation to previous works}
Several previous works have successfully used DMD to learn fast and accurate ROMs---see \cite{DMDKutz,MorDMD2014,ChristianDMD,DMDJian2020,randomDMD}, for instance. However, to the best of our knowledge, only the work in \cite{Lu2020} tailors DMD for hyperbolic PDEs. Ideologically, this work is similar to ours---it also tries to learn a ROM for a transformed solution and a de-transformer. However, a major difference is that it chooses (the de-transformer) $\td\varphi$ as the mapping to the Lagrangian coordinates and computes it by solving an ODE that governs the evolution of the characteristic curves. Following are the possible shortcomings of such an approach. Firstly, the resulting ODE is well posed only when the characteristic curves do not intersect i.e., when the solution does not exhibit spatial discontinuities. This limits the applicability of the method to solutions without spatial discontinuities. In contrast, the image registration technique that we use to compute $\td\varphi$ is well-equipped to handle space-time discontinuous solutions \cite{RegisterMOR,Taddei2020ST}. Secondly, computing Lagrangian coordinates requires the explicit form of the underlying PDE, which, in a data-driven setting, might not be available. Thirdly, for non-linear problems and multi-D spatial domains, tracing characteristics can be complicated. As a result, the examples in \cite{Lu2020} are mostly limited to 1D spatial domains. On the other hand, the image registration technique that we use is extendible to even complex multi-D spatial domains \cite{taddei2021Complex}.

The symmetry-reduction (SR) framework can be viewed as an inspiration behind our decomposition in \eqref{decompose u}---we refer to \cite{FrozenROM,Beyn2004,Rowley2001,Sapsis2018} for details of the SR framework. Roughly speaking, this framework decomposes $u$ as $\mcal H(\hat u)$, where $\mcal H$ is a continuous transformation that acts on the so-called template function $\hat u$. A finite-dimensional transformation $\mcal H$ results from a careful analysis of the underlying PDE, and is evolved using the so-called phase conditions. A PDE for the template function $\hat u$---derived using the equivariant property of the evolution operator stated below---is reduced using standard reduced order modelling techniques. The approximation for $\mcal H$ and the ROM for $\hat u$ are then recombined to approximate the solution $u$. 

One might be tempted to interpret the template function $\hat u$ as the transformed solution $g$ and the continuous transformation $\mcal H(g)$ as $g\circ \td\varphi$, and conclude that our decomposition is the same as that of SR. However, there are some key differences that one needs to consider. Firstly, the SR technique assumes that the evolution operator is equivariant i.e., it commutes with $\mcal H$. Only then can one derive an evolution equation for $\hat u$. In contrast, we make no such assumptions on the evolution operator. Our ROM for $\hat u$ is learned directly from the data and the learning technique does not require the equivariance property, or for that matter, any knowledge of the underlying dynamical system. As a result, our technique applies to a much broader class of evolution equations. Secondly, we compute both $\varphi$ and $\td\varphi$ in a purely data-driven manner without any prior explicit knowledge of the underlying hyperbolic PDE---recall that the SR framework carefully analyzes the underlying hyperbolic PDE to have a representation for $\mcal H$. Lastly, as already pointed out in \cite{RegisterMOR} and as the name suggests, SR decomposes the solution to reduce the symmetries in a dynamical system. We, on the other hand, decompose the solution solely from approximation considerations. 

Recent works indicate that auto-encoders can provide accurate non-linear reduced spaces for hyperbolic PDEs \cite{Fresca2020,KevinAuto,kim2020}. However, as noted in \cite{mojgani2020physicsaware}, an auto-encoder applied directly to the solution can result in an oscillatory approximation. Furthermore, despite a large training set, even short-time predictions can be inaccurate. Authors in \cite{mojgani2020physicsaware} circumvent both these problems by first transforming the snapshot matrix via a transformation of the space-time domain. An auto-encoder is then applied to the transformed snapshots and a long short-term memory (LSTM) architecture is used for time evolution. The takeaway is that even for modern artificial intelligence-based learning techniques, a decomposition of the form \eqref{decompose u} seems like a crucial first step---particularly when one seeks accurate predictions. The bounds on the latent space of an auto-encoder developed in \cite{franco2021deep} further elucidate the importance of (at least) Lipschitz continuity while applying auto-encoders. 

Note that by using POD for data compression (which is a type of an auto-encoder) and DMD for temporal evolution instead of LSTM, at an abstract level, one can relate our framework to that of \cite{mojgani2020physicsaware}. However, the work in \cite{mojgani2020physicsaware} is limited to 1D spatial domains. Furthermore, its efficiency, as compared to a HF solver, is unclear. Our approach on the other hand caters to multi-D spatial domains. Also, owing to the efficiency of the DMD approach, it is orders of magnitude faster than a HF solver. We will later empirically establish the efficiency of our method.

\subsection{Organization}
Rest of the article is organized as follows. \Cref{sec: DMD} presents the (plain) DMD framework in a general setting. \Cref{sec: ROM hyp} specializes the DMD framework from \Cref{sec: DMD} to hyperbolic PDEs. \Cref{sec: summary} summarizes our algorithm. \Cref{sec: num results} presents the numerical examples, and \Cref{sec: conclusion} closes the article with a conclusion.

 \section{Learning ROMs with DMD}\label{sec: DMD}
We briefly summarize the (plain) DMD approach for learning a ROM from snapshot data. We consider a general setting where we collect snapshots of a time-dependent vector $\mcal F(t)$ and learn a ROM for the same. Later in \Cref{sec: ROM hyp}, we choose $\mcal F(t)$ as the transformed solution and the de-transformer, thereby, specializing our framework to hyperbolic PDEs. We closely follow the formulation from \cite{DMDKutz,DMDTheory}, and as such, introduce no new formulations here. We re-iterate that only for the sake of completeness and ease of implementation, we use the DMD approach. In practise, any other learning techniques mentioned earlier would also suffice.

We acknowledge that since its inception, DMD has underwent numerous updates like the multi-resolution version \cite{mrDMD}; randomized version \cite{randomDMD}; extended version \cite{EDMD}; sparsity promoting version \cite{sparseDMD}, etc. Whether choosing a different DMD version improves the results that we report later is an interesting question in its own right. However, given the limited scope of our article, we refrain from comparing the different DMD versions here.

\subsection{Best-fit linear operator} Consider a set of discrete time instances 
\begin{gather}
\parSp_{train} := \{\ttrain{i}\}_{i=1,\dots,K}\subset [0,T], \label{train set}
\end{gather}
ordered such that 
$
0 = t_1 < t_2\dots < t_K = T.
$
For simplicity, we consider a constant time step size i.e., $t_k = t_1 + \Delta t\times (k-1)$, where $\Delta t>0$ denotes the time step size---extension to non-uniform time steps can be found in \cite{DMDTheory}. We consider a time-dependent vector $\mcal F(t)\in\mbb R^N$ and collect its snapshots in a snapshot matrix $\snapMat{\mcal F}{1}{K}\in\mbb R^{N\times K}$ given as 
\begin{gather}
\snapMat{\mcal F}{1}{K}:=\left(\mcal F(t_1),\dots,\mcal F(t_K)\right).
\end{gather}
These snapshots can result either from a dynamical system of the form $\mcal F(t_{k}) = \mcal L(\mcal F(t_{k-1}))$, or from experiments. Neither the explicit form of the evolution operator $\mcal L$, nor the details of the experiments, are required by the subsequent procedure. 

Consider a linear best-fit evolution operator $\td{\mcal A}:\mbb R^N\to\mbb R^N$, the action of which well approximates the discrete time trajectory of our snapshots. Equivalently,
\begin{gather}
\mcal F(t_{k + 1}) \approx \td{\mcal A}\mcal F(t_{k}). \label{best-fit}
\end{gather}
By representing the above expression in a matrix form, the linear operator can be computed by a least-squares best-fit given as
\begin{gather}
\td{\mcal A} := \argmin_{\mcal A^*\in\mbb R^{N\times N}}\|\snapMat{\mcal F}{2}{K}-\mcal A^*\snapMat{\mcal F}{1}{K-1}\|_F,
\end{gather}
where $\|\cdot\|_F$ represents the Frobenius norm of a matrix. Notice that $\snapMat{\mcal F}{2}{K}$ and $\snapMat{\mcal F}{1}{K-1}$ contain snapshots for the time intervals $[t_2,t_K]$ and $[t_1,t_{K-1}]$, respectively. Therefore, $\snapMat{\mcal F}{2}{K}\approx\td{\mcal A}\snapMat{\mcal F}{1}{K-1}$ is a matrix representation of the relation in \eqref{best-fit}. A solution to the above problem is given in terms of the pseudo-inverse $\snapMat{\mcal F}{1}{K-1}^\dagger$ and reads
\begin{gather}
\td{\mcal A} = \snapMat{\mcal F}{2}{K}\snapMat{\mcal F}{1}{K-1}^\dagger. \label{def A}
\end{gather}

\subsection{Reduced order model (ROM)}
We now develop a ROM for the vector $\mcal F(t)$. We approximate this vector in a set of POD modes and derive a governing equation for the POD coefficients. Owing to the linearity of our best-fit evolution operator, eventually, we will have an explicit expression for these POD coefficients. The details are as follows. 

Consider a low-rank singular value decomposition (SVD) based approximation of the snapshot matrix $\snapMat{\mcal F}{1}{K-1}$ given as 
\begin{gather}
\snapMat{\mcal F}{1}{K-1}  \approx \mcal U_n \Sigma_n \mcal V_n^*,\label{low-rank snapmat}
\end{gather}
where $(\cdot)^*$ represents the Hermitian transpose of a matrix. The matrix $\mcal U_n\in\mbb R^{N\times n}$ and $\mcal V_n\in\mbb R^{K\times n}$ contain the first $n\in\mbb N$ left and right singular vectors, respectively, of the snapshot matrix. Furthermore, $\Sigma_n\in\mbb R^{n\times n}$ is a diagonal matrix, which, at its diagonal, contains the first $n$ singular values of the snapshot matrix. Later, we will be more interested in understanding how the reduced order modelling error decays with $n$, and to this end, we will vary $n$ over a range of numbers. In practise, the value of $n$ can be chosen such that the relative error of the above approximation---which follows from the Schmidt-Eckart-Young theorem \cite{OptimalSVD}---stays below a given user-defined tolerance. The work in \cite{Rozza2008} outlines an implementation along these lines.

The columns of the matrix $\mcal U_n$ are the so-called POD modes. Approximating $\mcal F(t)$ in the span of these modes provides
\begin{gather}
\mcal F(t) \approx \mcal U_n\alpha_{\mcal F}(t),
\end{gather}
where $\alpha_{\mcal F}(t)\in\mbb R^n$ is a vector containing all the POD coefficients. A discrete-in-time evolution equation for this vector results from replacing the above approximation into our best-fit linear evolution equation \eqref{best-fit}. This provides
\begin{gather}
\alpha_{\mcal F}(t_{k+1}) = \td{\mcal A}_n\alpha_{\mcal F}(t_k)\hspB \text{where}\hspB \td{\mcal A}_n := \mcal U_n^* \td{\mcal A} \mcal U_n, \label{primitive POD}
\end{gather}
with the initial data $\alpha_{\mcal F}(t_1) := \mcal U_n^*\mcal F(t_1)$. The matrix $\td{\mcal A}_n$ is often referred to as the POD reduced matrix. 

With the eigendecomposition of $\td{\mcal A}_n$, we derive an explicit solution for the above equation.
Let $\{(\td\lambda_i,\td W_i)\}_{i=1,\dots,n}$ represent a pair of eigenvalues and eigenvectors of $\td{\mcal A}_n$. By repeatedly applying the operator $\td{\mcal A}_n$ to the projected initial data $\td b := (\td W)^\dagger \alpha_{\mcal F}(t_1)$, we recover an explicit expression for $\alpha_{\mcal F}(t)$ that reads
\begin{gather}
\alpha_{\mcal F}(t) = \sum_{i=1}^n \td W_i \exp(\td\omega_i t)\td b_i = \td W \exp(\td\omega t)\td b, \label{DMD naieve}
\end{gather}
where, for convenience, we define $\td{\omega}_i := \log(\td{\lambda}_i)/\Delta t$, and the diagonal matrix $\td\omega\in\mbb R^{n\times n}$ collects the different $\td\omega_i$ at its diagonal. Furthermore, the matrix $\td W\in\mbb R^{n\times n}$ contains the vectors $\td W_i$ as its columns. With the above expression, we recover a ROM for $\mcal F(t)$ that reads
\begin{gather}
\mcal F(t) \approx \mcal U_n\alpha_{\mcal F}(t) = \mcal U_n\td W \exp(\td \omega t)\td b.  \label{rom prim}
\end{gather}
Note that the above approximation is equation-free i.e., for any given $t\in [0,\infty)$, no further time stepping is required to approximate $\mcal F(t)$. 

\subsection{Practical considerations: DMD algorithm} In practice, two considerations hinder the applicability of the above approach. Firstly, the matrix $\td{\mcal A}$, computed via \eqref{best-fit}, can be severely ill-conditioned. Secondly, the value of $N$, which often denotes the dimensionality of a HF solver, can be extremely large. Under memory restrictions, it might be infeasible to store the $N\times N$ matrix $\td{\mcal A}$ and perform any subsequent algebraic manipulations. Note that in the above formulation, $\td{\mcal A}$ is required to compute the POD reduced matrix $\td{\mcal A}_n$.

The DMD algorithm proposed in \cite{DMDTheory} circumvents both the above issues by replacing the pseudo-inverse of the snapshot matrix $\snapMat{\mcal F}{1}{K-1}$, appearing in \eqref{best-fit}, by the pseudo-inverse of its low-rank SVD approximation given in \eqref{low-rank snapmat}. This replacement provides
\begin{gather}
\td{\mcal A} \approx \mcal A:= \snapMat{\mcal F}{2}{K}\mcal V_n\Sigma_n^{-1}\mcal U_n^*.
\end{gather}
Now, replacing $\td{\mcal A}$ by $\mcal A$ in the definition for $\td{\mcal A}_n$ given in \eqref{primitive POD}, we find 
\begin{equation}
\begin{aligned}
\td{\mcal A}_n\approx \mcal A_n := \mcal U_n^* \mcal A \mcal U_n = \mcal U_n^*\snapMat{\mcal F}{2}{K}\mcal V_n\Sigma_n^{-1}. 
\end{aligned}
\end{equation}
We emphasize that to compute $\mcal A_n$, one does not need to explicitly compute the $N\times N$ matrix $\mcal A$.

We approximate the eigendecomposition $(\td W,\td\lambda)$ of $\td{\mcal A}_n$ by the eigendecomposition $(W,\lambda)$ of $\mcal A_n$. Replacing this approximation in our ROM derived earlier in \eqref{rom prim}, we finally arrive at the DMD-based ROM given by
\begin{gather}
\mcal F(t) \approx \mcal F_n(t):= \mcal U_n W \exp(\omega t)b,  \label{rom DMD}
\end{gather}
where $b = (\mcal U_nW)^\dagger \mcal F(t_1)$, and $\omega := \log(\lambda)/\Delta t$.
Note that the columns of the matrix $\mcal U_n W$ are an approximation to the eigenvectors of the evolution operator $\td{\mcal A}$ and are the so-called DMD modes. These modes can be viewed as an approximation to the so-called Koopman modes of the infinite-dimensional, but linear, Koopman operator \cite{NonLinearFlows2009,DMDTheory}. Furthermore, the eigendecomposition $(\mcal U_n W,\lambda)$ is what one refers to as the DMD of the best-fit linear operator $\td{\mcal A}$.

\section{Learning ROMs for hyperbolic PDEs}\label{sec: ROM hyp}
This section specializes the DMD approach discussed earlier to hyperbolic PDEs. We start with some preliminaries.
\subsection{Preliminaries}
A general hyperbolic PDE is given as 
\begin{equation}
\begin{gathered}
\pd_t u + \nabla \cdot f(u) = 0,\hsp \text{on}\hsp \Omega\times [0,T],\hspB u = u_0,\hsp\text{on}\hsp \Omega\times \{0\},\\
u = u_{bc},\hsp \text{on}\hsp \pd\Omega\times [0,T],
\end{gathered}
\end{equation}
where $u_0$ and $u_{bc}$ represent the initial and the boundary data, respectively. We interpret the boundary conditions in the weak sense of \cite{FlochBC}. For simplicity, the solution $\sol{t}$ is assumed to be scalar-valued. An extension to vector-values functions is straightforward; to corroborate this claim, later, we will consider a numerical example with a vector-valued solution. The time instance $T > 0$ is some finite time horizon, and $\Omega\subset\mbb R^d$ is the spatial domain. The function $f:\mbb R\to\mbb R^d$ is the so-called flux-function, with $\nabla$ being a gradient operator in $\mbb R^d$. 

We compute a HF approximation of the solution via a second-order finite volume scheme equipped with the local Lax Friedrich flux and the van Leer flux limiter \cite{vanLeer1974}. For time-stepping, we consider an explicit second-order Runge--Kutta scheme with the CFL number set to $0.5$.
 We denote the HF approximation by
\begin{gather}
\sol{t}\approx \solFV{t}\in\mcal X_N\subset L^2(\Omega).
\end{gather}
where $\opn{dim}(\mcal X_N) = N$. The space $\mcal X_N$ is a finite-volume space of piecewise constant functions defined over a shape-conforming and regular spatial discretization of $\Omega$ given as 
\begin{gather}
\Omega = \bigcup_{i=1}^N\mcal I_i,
\end{gather}
where $\mcal I_i$ represents the $i$-th grid cell. With $\Pi : L^2(\Omega)\to \mcal X_N$ we represent the orthogonal $L^2$ projection operator. For the sake of completeness, we have chosen a finite-volume HF solver. In practise, any other HF solver would also suffice.

For later convenience, with 
\begin{gather}
\td{\mcal X}_N\subset L^2(\Omega)
\end{gather}
 we represent a space of piecewise linear and continuous finite-element functions defined over the spatial mesh. The space $\td{\mcal X}_N$ is $\td N$ dimensional with $\td N$ being the number of vertices in the spatial mesh. With $\td\Pi : C^0(\bar{\Omega})\to\td{\mcal X}_N$ we represent a projection operator defined as 
 \begin{gather}
 \td\Pi h (\hat x_i) := h (\hat x_i),\hspB \forall i\in \{1,\dots,\td N\}, \label{def td Pi}
 \end{gather}
 where $\{\hat x_i\}_i$ denotes the set of vertices of the spatial mesh. Note that because $\td{\mcal X}_N$ contains piecewise linear and continuous functions, the above relation uniquely defines $\td\Pi h$.

\subsection{Decompose then learn approach}
As stated earlier, solutions to hyperbolic PDEs exhibit temporal discontinuities and a ROM learned directly for the solution can be inefficient. Therefore, we propose to decompose the HF solution as 
\begin{gather}
\solFV{t} = g_N(\td \varphi(\cdot,t),t)\hspB \text{where}\hspB g_N(\cdot,t) := u_N(\varphi(\cdot,t),t). \label{decompose uN}
\end{gather}
Recall that $\td\varphi(\cdot,t) = \varphi(\cdot,t)^{-1}$.
Our idea was to first learn a ROM for $g_N(\cdot,t)$ and $\td{\varphi}(\cdot,t)$, followed by recovery of a ROM for $\solFV{t}$ via the above relation---refer back to \Cref{fig: learn ROM} for a summary. Throughout this subsection, we assume that, at the time instances $\parSp_{train}$ given in \eqref{train set}, we have access to both, the snapshots of the HF solution $\{\solFV{t}\}_{t\in\parSp_{train}}$ and the snapshots of the spatial transform $\{\varphi(\cdot,t)\}_{t\in\parSp_{train}}$. The former results from our HF numerical solver, whereas for the latter, we use the image-registration technique outlined later in \Cref{sec: snapshots varphi}.

\subsubsection{ROM for $g_N$ and $\td\varphi$}
We specialize the DMD framework from \Cref{sec: DMD} to learn a ROM for $g_N(\cdot,t)$ and $\td\varphi(\cdot,t)$. 
To proceed further, we need finite dimensional HF representations for $g_N(\cdot,t)$ and $\td\varphi(\cdot,t)$. We use the projection operator $\Pi$ to project $g_N(\cdot,t)$ onto $\mcal X_N$ and collect the resulting degrees of freedom in a vector $G(t)\in\mbb R^N$. Likewise, since we expect $\varphi(\cdot,t)$ to be a diffeomorphism (explained later) it is reasonable to project $\td\varphi(\cdot,t)$ onto the finite-element space $\td{\mcal X}_n$ using the projection operator $\td\Pi$ given in \eqref{def td Pi}. We compute the value $\td\varphi(\hat x_i,t)$ at the $i$-th vertex $\hat x_i$  using the non-linear least-squares problem given as
\begin{gather}
\td\varphi(\hat x_i,t) := \argmin_{x\in\Omega}\|\varphi(x,t)-\hat x_i\|^2_2,
\end{gather}
which we solve using the \texttt{lsqnonlin} function of MATLAB. We collect the degrees of freedom of $\td\Pi \td\varphi(\cdot,t)$ in the vector $\td\Phi(t)\in\mbb R^{\td N\times d}$. 

Recall that the framework in \Cref{sec: DMD} provides a ROM for a time-dependent vector $\mcal F(t)$. In this framework, we now choose $\mcal F \equiv G$ and $\mcal F\equiv \td\Phi$ and recover the following ROMs
\begin{equation}
\begin{aligned}
G(t)\approx G_n(t) :=&\mcal U_n^G W^G \exp(\omega^G t)b^G,\\
\td\Phi(t)\approx \td\Phi_n(t) :=&\mcal U_n^{\td\Phi} W^{\td\Phi} \exp(\omega^{\td\Phi} t)b^{\td\Phi}.
\end{aligned}
\end{equation}
For completeness, we briefly recall the different objects appearing in the expression for $G_n(t)$---similar definition holds for the terms appearing in $\td\Phi_n(t)$. The matrices $\mcal U_n^G$ and $\mcal V_n^G$ contain the first $n$ left and right singular vectors, respectively, of the snapshot matrix
\begin{gather}
\snapMat{G}{1}{K-1} := \left(G(t_1),\dots, G(t_{K-1})\right)). \label{snapmat G}
\end{gather}
Furthermore, the diagonal matrix $\Sigma^G_n$, at its diagonal, contain the first $n$ singular values of the above matrix.
The pair $(W^G,\lambda^G)$ is the eigendecomposition of the POD reduced matrix
\begin{gather}
\mcal A_n^G:=\left(\mcal U_n^G\right)^*\snapMat{G}{2}{K}\mcal V^G_n\left(\Sigma_n^G\right)^{-1}, 
\end{gather}
and $\omega^G := \log(\lambda^G)/\Delta t$. The matrix $\mcal U_n^G W^G$ contains, as its columns, the DMD modes for $G(t)$. Lastly, the amplitude vector $b^G$ results from projecting the initial data onto the DMD modes and reads $b^G := (\mcal U_n^G W^G)^\dagger G(t_1)$.

For further convenience, we express the above ROMs in functional forms.
Let $g_n(\cdot,t)$ and $\td\varphi_n(\cdot,t)$ represent functions in $\mcal X_N$ and $[\td{\mcal X}_N]^{d}$, respectively, whose degrees of freedom are uniquely identified by the vectors $G_n(t)$ and $\td\Phi_n(t)$, respectively. Then, we have the reduced approximation
\begin{gather}
g_N(\cdot,t)\approx g_n(\cdot,t),\hspB \td\varphi(\cdot,t)\approx \td\varphi_n(\cdot,t).
\end{gather}

\begin{remark}[POD modes for $g_N$ and $\td\varphi$]
To balance accuracy with computational cost, one can consider different number of POD modes to approximate $g_N(\cdot,t)$ and $\td\varphi(\cdot,t)$. Such an approach would require an a-posteriori error indicator, which will be a part of our future investigations.
\end{remark}
\subsubsection{ROM for $u_N$}
Replacing the above reduced approximations in the decomposition for $\solFV{t}$ given in \eqref{decompose uN}, we finally recover a ROM for $\solFV{t}$ that reads
\begin{gather}
\solFV{t} \approx u_n(\cdot,t) := g_n(\td\varphi_n(\cdot,t),t). \label{ROM uN}
\end{gather}

We now consider the following question: Given that $g_n$ and $\td\varphi_n$ are accurate approximations (in some sense) of $g_N(\cdot,t)$ and $\td\varphi(\cdot,t)$, respectively, how well does $u_n(\cdot,t)$ approximates $\solFV{t}$? The result below provides an answer to some extent. In essence, the result establishes that
\begin{gather}
\|u_N-u_n\|_{L^{\infty}([0,T];L^1(\Omega))} = \mcal O(\delta),
\end{gather}
where $\delta$ is the error in approximating $g_n$ and $\td\varphi_n$. Thus, in terms of scaling with $\delta$, the ROM $u_n$ performs as good as the ROMs $g_n$ and $\td\varphi_n$. Later, we further elaborate upon the result. First, let us summarize the result below. For notational simplicity, we consider a 1D spatial domain i.e., $d=1$. A similar result can be obtained for multi-D problems.

\begin{theorem}\label{lemma: err bound}
Let $d=1$. Assume that for all $t\in [0,T]$, $\varphi(\cdot,t)$, $\td\varphi(\cdot,t)$ and $\td\varphi_n(\cdot,t)$ are $W^{1,\infty}(\Omega)$ homeomorphisms, where $W^{1,\infty}(\Omega)$ is a Sobolev space of functions with bounded weak-derivatives upto the first-order. Let $n\in \mbb N$ be such that for a $\delta \geq 0$ it holds
\begin{gather}
\|g_N-g_n\|_{L^\infty([0,T]; L^1(\Omega))}\leq \delta\hspB\text{and}\hspB \|\td\varphi-\td\varphi_n\|_{L^{\infty}([0,T];W^{1,\infty}(\Omega))}\leq \delta.
\end{gather}
Then,  the error $\|u_N-u_n\|_{L^{\infty}([0,T];L^1(\Omega))}$ satisfies
\begin{equation}
\begin{aligned}
\|u_N-u_n\|_{L^{\infty}([0,T];L^1(\Omega))} \leq C_1(u,\td \varphi)\delta +\delta^2. \label{err bound}
\end{aligned}
\end{equation}
Above, $C_1(u,\td \varphi):= (\|u\|_{L^{\infty}([0,T];BV(\Omega))} + \|\nabla \td\varphi\|_{L^{\infty}([0,T]\times\Omega)})$, $|\Omega|$ represents the Lebesgue measure of $\Omega$, and $BV(\Omega)$ is a space of functions with bounded variations.
\end{theorem}
\begin{proof}
See \Cref{app: error bound}.
\end{proof}
Remarks related to the above result are in order.
\begin{enumerate}
\item\textbf{Diffeomorphic $\varphi$:} Empirically, the image registration technique that we outline later provides a diffeomorphic $\varphi(\cdot,t)$ when the maximum displacement $\sup_{x\in\Omega}|\varphi(x,t)-x|$ is not \textit{too large}---see \cite{Amit1994} for further details. In the context of image registration, these are the so-called small-displacement problems. As we demonstrate later, even for this sub-class of problems, our method significantly outperforms the standard DMD approach. However, for problems with large displacements, our registration technique is inadequate and should be replaced with the framework proposed in \cite{LDDM}. We plan to undertake such an approach in the future.

\item\textbf{Gradient of $\td\varphi$:} The size of the Jacobian $\|\nabla \td\varphi\|_{L^{\infty}}$ should not be too large. Otherwise, because of the term $C_1(u,\td\varphi)$ appearing in the error bound \eqref{err bound}, our ROM for $\solFV{t}$ might be inaccurate. As shown in \cite{SarnaCalib2020}, $\|\nabla \td\varphi\|_{L^{\infty}}$ is small if the spatial discontinuities in $\solFV{t}$ do not move too far away from the discontinuities in a reference solution $\solFV{\tref}$, where $\tref\in [0,T]$. Later in \Cref{sec: snapshots varphi} we will further elaborate on this point. 

\item\textbf{Homeomorphic $\td\varphi_n$:} We assume that the ROM $\td\varphi_n(\cdot,t)$ is a homeomorphism. We acknowledge that this is a difficult property to guarantee. Indeed, all of the previous works in \cite{Welper2017,Nair,RegisterMOR,MATS,Cagniart2019}, which try to either approximate or generalize (via regression) a spatial transform, cannot guarantee this property. The reason being that we (and all of these previous works) compute $\td\varphi_n(\cdot,t)$ in a linear reduced space. In our context, the linear reduced space is spanned by the POD modes and is not a subset of all the homeomorphisms over $\Omega$.\footnote{A span of POD modes contains the zero element which is not a homeomorphism from $\Omega$ to $\Omega$.} Therefore, without any additional constraints on the POD coefficients, a POD approximation might not be a homeomorphism over $\Omega$. Nevertheless, for an $n\ll N$, the linear approximation $\td\varphi_n$ is still accurate in the $L^p$ (or $W^{p,\infty}$) sense. Furthermore, it is cheap and easy to compute, and in our experience, if the condition on the movement of the discontinuities stated in the previous point is satisfied then, it is a homeomorphism with well-behaved derivatives. We will elaborate further with numerical examples.

\item\textbf{Generality of the error bound:} The error bound in \eqref{err bound} is general in the sense that it is independent of the technique used to compute snapshots of the spatial transform $\varphi$. As a result, it is also valid for a registration technique different from ours---the optimal transport-based \cite{RimOT} or the Lagrangian coordinates based \cite{MojganiLagrangian,Lu2020} techniques being two noteworthy examples. 

\end{enumerate}

\subsection{Snapshots of $\varphi$}\label{sec: snapshots varphi} To realize the above technique, we require the snapshots $\{\varphi(\cdot,t)\}_{t\in\parSp_{train}}$. We consider an image registration technique to compute these snapshots. The technique is outlined below, we refer to \cite{RegisterMOR,sarna2021datadriven} for additional details. We restrict our discussion (and our numerical examples) to a square spatial domain $\Omega = [0,1]^2$. The subsequent procedure can be generalized to a class of curved spatial domains by partitioning it into subsets, followed by mapping each of these subsets onto a unit square \cite{taddei2021Complex}.

Consider the splitting
\begin{gather}
\Psi(\cdot,t):= \opn{Id} - \varphi(\cdot,t), \label{split varphi}
\end{gather}
where $\Psi(\cdot,t)$ is the displacement field, and $\opn{Id}(x) := x$. We consider a $\Psi(\cdot,t)$ that belongs to the space $\mcal P_M:=[\mbb P_M]^2$ that reads
\begin{gather}
\mbb P_M:=\opn{span}\{L_{jk}\Upsilon\}_{j,k=1,\dots,M}, \label{def PM}
\end{gather}
where $L_{ij}(x)$ is a product of the $i$-th and $j$-th Legendre polynomials $l_i(x_1)$ and $l_j(x_2)$, respectively, and $\Upsilon(x):=\Pi_{m=1}^d x_m(1-x_m)$. By the definition of $\Upsilon$, we find the boundary condition $\Psi(x,t) = 0$. Note that this is a stronger version of the boundary conditions than that proposed in Proposition-2.3 of \cite{RegisterMOR}. At least for the test cases we consider, these boundary conditions provide reasonable results without any additional constraints on the Jacobian of $\varphi$. One may also consider periodic boundary conditions with a Fourier series expansion for $\Psi(\cdot,t)$---see \cite{Amit1994}.

We compute $\Psi(\cdot,t)$ using the minimization problem
\begin{gather}
\Psi(\cdot,t) = \argmin_{\Psi^*\in\mcal P_M}\left(\mcal M(\Psi^*,t,\tref)) + \mcal R(\Psi^*)\right),
\end{gather}
where $\mcal M(\cdot,\cdot)$ is a so-called matching criterion, and the regularization term $\mcal R(\cdot)$ penalizes the spatial regularity of $\Psi$, which promotes a diffeomorphic $\varphi(\cdot,t)$ and provides numerical stability \cite{ReviewDeformReg,Klein2007}. The matching criterion $\mcal M$ should be such that the transformed solution $g_N(\cdot,t)$, at least ideally, has no discontinuities along the temporal domain and consequently, as compared to $\solFV{t}$, has much better approximability in a linear reduced space. Following the empirical success reported in \cite{Nair,RegisterMOR,Welper2017}, we set 
\begin{equation}
\begin{gathered}
\mcal M(\Psi^*,t,\tref) := \|u_N(\opn{Id}+\Psi^*,t)-\solFV{\tref}\|^2_{L^2(\Omega)},\hspB
\text{and}\hspB \mcal R(\Psi^*):=\epsilon\|\Delta \Psi^*\|^2_{L^2(\Omega;\mbb R^2)}.
\end{gathered}
\end{equation}
Following the previous works in \cite{RegisterMOR,Taddei2020ST,sarna2021datadriven}, we set $\epsilon = 10^{-3}$. We compute $M$ iteratively using the technique outlined in Section-3.2.1 of \cite{sarna2021datadriven}---we do not repeat the details here for brevity. Furthermore, the initial guess is set using the re-ordering scheme proposed in Section-3.1.2 of \cite{RegisterMOR}. To perform optimization, we resort to the \texttt{fmincon} routine of MATLAB.

The choice of the reference time-instance $\tref\in [0,T]$ is key in computing an accurate spatial transform. We consider a fixed time-invariant $\tref$ and set $\tref = T/2$. We expect such a choice to be accurate for problems where: (i) the discontinuity topology i.e., the number and the relative orientation of the discontinuities, does not change; and (ii) relative to the discontinuities in $\solFV{\tref}$, the discontinuities in any other solution do not move too close or too far away from each other. In case these conditions are violated, one should try to satisfy them locally by partitioning the time domain into subsets, followed by a local-in-time choice for $\tref$---see  \cite{WelperAdaptive,SarnaCalib2020}. Akin to the framework in \cite{NonLinearMOR}, this procedure will result in a local-in-time reduced space, which will adapt to the changing discontinuity topology. We plan to develop such an approach in the future. For now---similar to \cite{Welper2017,Nair}---we assume that the above two conditions are satisfied. Also note that these two conditions ensure that the gradients of $\varphi$ and $\td\varphi$ are well behaved, which, as noted in \Cref{lemma: err bound}, is desirable.

\begin{remark}[Approximability of $\td\varphi(\cdot,t)$ in the POD basis]
For our ROM to be accurate, both the transformed solution $g_N(\cdot,t)$ and the de-transformer $\td\varphi(\cdot,t)$ should be well approximable in the POD basis. Assuming that our registration technique works well, we expect $g_N(\cdot,t)$ to be well approximable in the POD basis. For $\td\varphi(\cdot,t)$ to be sufficiently regular in time and thus (hopefully) well approximable in the POD basis, the spatial discontinuities in the solution $\sol{t}$ should be sufficiently regular in time. Precisely, if $\sol{t}$ is discontinuous along the curve $\mcal D(t)\subset\Omega$ then, we want $\mcal D(\cdot)$ to be sufficiently regular---the precise quantification of the required regularity and the condition that makes $\mcal D(\cdot)$ regular is given in Corollary-3.4 of \cite{SarnaCalib2020}. We insist that for several problems of practical interest, and particularly for our numerical examples, this condition is indeed satisfied.
\end{remark}
\subsection{Why learn a dynamical system for $\td\varphi$?}
Assume that the solution $\sol{t}$ satisfies a dynamical system of the form $\pd_t u = \mcal L(u)$, with $\mcal L$ being some differential operator in space. Then, one can check that the transformed solution $g(\cdot,t)$ satisfies $\pd_t g = \mcal L_{\varphi}(g)$, with $\mcal L_{\varphi}$ being a transformed differential operator that depends upon $\varphi$---see \cite{RegisterMOR,Taddei2020ST} for some examples. This justifies our motivation behind learning a dynamical system (followed by learning a ROM) for $g$. However, it is not obvious why $\td\varphi$ would satisfy a dynamical system. Actually, if $\td\varphi(\cdot,t)$ is a diffeomorphism and if for all $x\in\Omega$, $\td\varphi(x,\cdot)\in C^1([0,T])$ then, one can show that it satisfies a dynamical system of the form
\begin{gather}
\pd_t \td\varphi(x,t) = v(\varphi(x,t),t)\hspB\forall (x,t)\in \Omega\times [0,T],
\end{gather}
where $v\in C^1([0,T]\times \bar\Omega,\mbb R^d)$ is a velocity field induced by $\td\varphi(x,t)$. The above result follows from the Cauchy--Lipschiz theorem and is well studied in the image registration literature---see \cite{LDDM,Trouve2002}, for instance. 

A few comments related to the above equation are in order.
\begin{enumerate}

\item \textbf{Relation to image registration:} For a certain class of image registration techniques, the velocity field results from an optimization problem, and the snapshots of the spatial transform are computed using the above ODE \cite{LDDM}. Our set-up is the exact opposite. Using the snapshots of $\td\varphi$, we learn (and then reduce) the dynamics of the above ODE.

\item \textbf{Relation to Lagrangian DMD:} Choosing $v(x,t)$ as the characteristic velocity of the hyperbolic PDE and computing $\td\varphi$ using the above ODE, we recover the Lagrangian DMD approach developed in \cite{Lu2020}. However, as explained earlier in the introduction, the resulting ODE is well-posed only for non-intersecting characteristic curves, which, for non-linear problems, limits the applicability of the Lagrangian DMD approach to space-time smooth solutions. Furthermore, for both linear and non-linear problems, to derive an explicit expression for the characteristic velocity, we necessarily require the explicit form of the underlying PDE, which might be unavailable in a data-driven setting.
\end{enumerate}

\subsection{The standard DMD approach}\label{sec: standard DMD}
We briefly recall the standard DMD approach since we use it in our numerical experiments. Let $U(t)\in\mbb R^N$ represent a vector that contains the degrees of freedom of $\solFV{t}$. In the standard approach, one does not perform any solution decomposition and directly chooses $\mcal F\equiv U$ in the framework developed earlier in \Cref{sec: DMD}. This results in a ROM that reads \cite{DMDKutz}
\begin{gather}
U(t)\approx U_n(t) :=\mcal U_n^U W^U \exp(\omega^U t)b^U, \label{approx u DMD}
\end{gather}
where all the matrices are as defined earlier in \Cref{sec: DMD}. For further discussion, we recall that $\mcal {U}_n^U$ contains the first $n$ left singular vectors of a snapshot matrix given as 
\begin{gather}
\snapMat{U}{1}{K-1} := \left(U(t_1),\dots, U(t_{K-1})\right). \label{snapmat U}
\end{gather}

\subsubsection{Interpolation and extrapolation regime}
The above ROM can either be used in the interpolation regime $t\in [0,T]$ or in the extrapolation regime $t > T$. For hyperbolic PDEs, extrapolation is a much more difficult task than interpolation. One can easily find examples where a ROM that is convergent (with $n$) in the interpolation regime might fail to converge in the extrapolation regime. Let us elaborate further. 

Assuming that the training set $\parSp_{train}$ is dense in $[0,T]$, in the interpolation regime, we expect to find an $n\in\mbb N$ to meet any given tolerance $\delta >0$ on the error $\sup_{t\in [0,T]}\|U(t)-U_n(t)\|_{l^2}$. Of course, if $\sol{t}$ has temporal discontinuities, which is usually the case, the value of $n$ might scale poorly with $\delta$; for instance, for the 1D advection equation with a unit step function as the initial data, $n$ scales as $\mcal O(1/\delta ^2)$---see \cite{PeterBook} for a proof. Nevertheless, at least there exists an $n$---no matter how large---such that one can expect to meet a given error tolerance. Equivalently, we expect the reduced approximation to converge in the interpolation regime---empirical results in \cite{DMDKutz} further corroborate our claim.

In the extrapolation regime, irrespective of the number of snapshots collected, our reduced approximation might not converge. There are a plethora of examples where the solution is orthogonal to all other solutions at previous time instances---see \cite{MATS,RimTR,Lu2020,RBBook}. Equivalently, for $T < t_0 \leq t$ and for all $t^*\in [0,T]$, we have $\lan U(t),U(t^*)\ran_{l^2} = 0$. As a result, for all $n\in\mbb N$ and for all $t_0 < t$, the solution $U(t)$ is orthogonal to the approximation space $\opn{range}(\mcal U_n)$, resulting in no convergence. We emphasize that the orthogonality condition can also hold for space-time smooth solutions, which, we expect, are well approximable in the interpolation regime---\cite{Lu2020} provides an example. Furthermore, even if the solution is not orthogonal to $\opn{range}(\mcal U_n)$, an approximation in $\opn{range}(\mcal U_n)$ can wrongly predict the shock locations in $\solFV{t}$. This usually results in large error values. In our numerical experiments, the standard DMD approach will perform poorly in the extrapolation regime. The above discussion explains our findings to some extent.

In contrast to the solution, the transformed solution $g_N$ exhibits minimal transport-dominated behaviour. As a result, it is well approximable both in the interpolation and the extrapolation regime---a similar comment holds for the de-transformer $\td\varphi$. However, for both the standard DMD and our new approach, as shown in the next section, the error in the extrapolation regime grows monotonically with time. Thus, despite the de and re-composition, for hyperbolic PDEs, accurate long-time prediction remains a challenging task---authors in \cite{Lu2020} make similar observations. Let us recall that a DMD-based ROM for non-linear parabolic and elliptic problems faces similar challenges \cite{lu2020prediction}.

 \section{Summary}\label{sec: summary}
In \Cref{summary: offline} and \Cref{summary: online}, we summarize the offline and the online stages of the algorithm, respectively. The offline phase relies on the \texttt{get_DMD_matrices} routine that we summarize in \Cref{summary: DMD}. This routine provides the different matrices required to compute a DMD-based ROM. 
\begin{algorithm}[ht!]
\caption{\texttt{get_DMD_matrices} routine}
\begin{algorithmic}[1] \label{summary: DMD}
%% initliazation
\STATE \textbf{Input: } $n$, $\mcal S_{1,K}$, $\Delta t$
\STATE \textbf{Output: }$\mcal U_n ,$ $W,$ $\omega$, $b$
\STATE $\left[\mcal U_n,\Sigma_n,\mcal V_n\right]\leftarrow \texttt{svds}(\mcal S_{1,K-1},n)$\hsp\COMMENT{\texttt{svds} is the in-built MATLAB routine.}
\STATE $\mcal A_n \leftarrow \mcal U_n^*\mcal S_{2,K}\mcal V_n\Sigma_n^{-1}$ \hsp\COMMENT{$\mcal A_n$ is the POD reduced matrix.}
\STATE $\left[W,\lambda\right]\leftarrow \texttt{eig}(\mcal A_n)$ \hsp\COMMENT{\texttt{eig} is the in-built MATLAB routine.}
\STATE $\omega\leftarrow \log(\lambda)/\Delta t$, $b\leftarrow (\mcal U_n W)^\dagger s_1$ \hsp\COMMENT{$s_1$ is the first columns of $\mcal S_{1,K}$.}
\end{algorithmic}
\end{algorithm}

\begin{algorithm}[ht!]
\caption{Summary: Offline phase}
\begin{algorithmic}[1] \label{summary: offline}
%% initliazation
\STATE \textbf{Input: } $n$, $\mcal S_{1,K}(G)$, $\mcal S_{1,K}(\td\Phi)$, $\Delta t$
\STATE \textbf{Output: }$\left\{\mcal U_n^G ,W^G,\omega^G, b^G\right\}$, $\left\{\mcal U_n^{\td\Phi} ,W^{\td\Phi},\omega^{\td\Phi}, b^{\td\Phi}\right\}$
\STATE $\left[\mcal U_n^G ,W^G,\omega^G, b^G\right]\leftarrow \texttt{get_DMD_matrices}(n,\snapMat{G}{1}{K},\Delta t)$
\STATE $\left[\mcal U_n^{\td\Phi} ,W^{\td\Phi},\omega^{\td\Phi}, b^{\td\Phi}\right]\leftarrow \texttt{get_DMD_matrices}(n,\snapMat{\td\Phi}{1}{K},\Delta t)$
\end{algorithmic}
\end{algorithm}

\begin{algorithm}[ht!]
\caption{Summary: Online phase}
\begin{algorithmic}[1] \label{summary: online}
%% initliazation
\STATE \textbf{Input: }$\left\{\mcal U_n^G ,W^G,\omega^G, b^G\right\}$, $\left\{\mcal U_n^{\td\Phi} ,W^{\td\Phi},\omega^{\td\Phi}, b^{\td\Phi}\right\}$, $t$
\STATE \textbf{Output: }$u_n(\cdot,t)$
\STATE $G_n(t) \leftarrow \mcal U_n^G W^G \exp(\omega^G t)b^G$
\STATE $\td \Phi_n(t) \leftarrow \mcal U_n^{\td\Phi} W^{\td\Phi} \exp(\omega^{\td\Phi} t)b^{\td\Phi}$
\STATE $g_n(\cdot,t)\leftarrow G_n(t)$, $\td\varphi_n(\cdot,t)\leftarrow \td\Phi_n(t)$
\STATE $u_n(\cdot,t) = g_n(\td\varphi_n(\cdot,t),t)$
\end{algorithmic}
\end{algorithm}

\section{Numerical Results}\label{sec: num results}
We abbreviate our DMD approach as TS--DMD, where TS stands for transformed snapshots. The standard approach (outlined in \Cref{sec: standard DMD}) is abbreviated as DMD. Also recall that we abbreviate the high-fidelity solver as HF. 

\subsection{Description of test case} Following is a description of the test cases we consider.
\begin{enumerate}
\item \textbf{Test-1 (1D advection)} We consider the 1D advection equation
\begin{gather}
\pd_tu + \pd_x u  = 0,\hsp \text{on} \hsp \Omega\times[0,T],\hspB u = u_0,\hsp \text{on}\hsp\Omega\times \{0\}.
\end{gather}
We set $\Omega = (-0.2,2)$ and $T = 0.8$. The initial data $u_0$ reads
\begin{gather}
u_0 = \chi_{[\delta,\delta + 0.5]},
\end{gather}
where $\chi_A$ represents a characteristic function of the set $A\subset \mbb R$. We set $\delta := -0.2$. Along the boundary $\pd\Omega\times [0,T]$, we set $u = 0$.
\item \textbf{Test-2 (1D wave equation):} We consider the 1D wave equation (rewritten as a first order system)
    \begin{gather}
        \pd_t u + A \pd_x u = 0,\hsp\text{on}\hsp \Omega\times [0,T], \label{wave equation}
    \end{gather}
    where $u = (u_1,u_2)^T$ is the vector-valued solution, and the matrix $A$ reads
    \begin{gather}
    A := \left(\begin{array}{c c}
    0 & 1 \\ 
    1 & 0
    \end{array}\right). 
    \end{gather}
    We choose $\Omega = (-0.3,3)$, and $T = 0.6$. The initial data is given by the linear superposition
    \begin{equation}
    \begin{gathered}
   u_1(x,t=0) =  \frac{1}{\sqrt{2}}\left(w_1(x) + w_2(x)\right),\\ u_2(x,t=0) = \frac{1}{\sqrt{2}}\left(-w_1(x) + w_2(x)\right),
    \end{gathered}
    \end{equation}
    where $w_1$ and $w_2$ read
    \begin{equation}
    \begin{aligned}
    w_1(x) := &(\sin(2\pi (x + 0.2)) + 1)\chi_{[\delta_1-0.5,\delta_1]}(x),\\
     w_2(x) := &(\sin(2\pi (x-2.3)) + 1)\chi_{[\delta_2-0.5,\delta_2]}(x). \label{def: sin bumps}
    \end{aligned}
    \end{equation}
    We set $\delta_1:=0.3$, and $\delta_2 := 2.8$. Along the boundary $\pd\Omega\times [0,T]$, we prescribe $u = 0$. 
 \item \textbf{Test-3 (2D Burgers' equation):}  We consider the 2D Burgers' equation given as 
    \begin{gather}
    \pd_t u(x,t) + \left(\frac{1}{2},\frac{1}{2}\right)^T\cdot\nabla u(x,t)^2=0,\hsp \forall (x,t)\in\Omega\times [0,T].
    \end{gather}
	The initial data $u_0$ reads
    \begin{gather}
    u_0 = \chi_{[0,0.5]^2}. \label{ic burgers}
    \end{gather}
    We set $\Omega = (-0.1,1.4)^2$, and $T=1$. Along the boundary $\pd\Omega\times \parSp$, we prescribe $u = 0$. 
 
\end{enumerate}
\subsection{Error quantification} For a given time-instance $t\in [0,\infty)$, we define the relative $L^1(\Omega)$ error as 
\begin{gather}
E(n,t) := \frac{1}{\|\solFV{t}\|_{L^1(\Omega)}}\left(\|\solFV{t}-\solROM{n}{t}\|_{L^1(\Omega)}\right), \label{def E}
\end{gather}
where $\solFV{t}$ is the HF solution, and the reduced solution $\solROM{n}{t}$ can either result from DMD or TS--DMD. The average error inside a time-interval $[t_0,t_1]$ is given by
\begin{gather}
E_a(n,t_0,t_1) := \frac{1}{\#\parSp_{test}}\sum_{t\in\parSp_{test}}E(n,t), \label{def Ea}
\end{gather}
where $\parSp_{test}\subset [t_0,t_1]$ contains $100$ uniformly and independently sampled parameters from the time interval $[t_0,t_1]$. The average error defined as such will allows us to separately study the error in the interpolation $t\in [0,T]$ and the extrapolation $t > T$ regime. Recall that we expect DMD to perform poorly in the extrapolation regime---see \Cref{sec: standard DMD}. The following results will further corroborate our expectations.

\subsection{Test-1} We partition $\Omega$ into $4\times 10^3$ uniform elements, and collect snapshots at $500$ uniformly placed time instances inside $[0,0.8]$. We approximate the displacement field $\dev{\tref}{t}$ (defined in \eqref{split varphi}) in the polynomial space $\mcal P_{M=4}$, where the value of $M$ results from the procedure outlined in \Cref{sec: snapshots varphi}.
\subsubsection{Singular value decay} For further convenience, we recall that:
\begin{enumerate}
\item $\snapMat{U}{1}{K}$ contains the (untransformed) snapshots $\{\solFV{t}\}_{t\in\ttrain{train}}$;
\item $\snapMat{G}{1}{K}$ contains the transformed snapshots $\{g_N(\cdot,t)\}_{t\in\ttrain{train}}$;
\item $\snapMat{\td\Phi}{1}{K}$ contains the de-transformer's snapshots $\{\td\varphi(\cdot,t)\}_{t\in\ttrain{train}}$.
\end{enumerate}
For these three snapshot matrices, we compare the singular value decay. We denote the $n$-th singular value by $\sigma_n$, and study the scaled singular value $\sigma_n/\sigma_1$.

 \Cref{fig: test-1 sigma} compares the singular value decay. The matrix $\snapMat{U}{1}{K}$ has a slow singular value decay, indicating that a large number of POD modes (i.e., the value of $n$ in \eqref{approx u DMD}) might be required to well approximate the solution. This slow decay is triggered by the temporal discontinuities in the solution---see \Cref{fig: test-1 snapmat}. In contrast, as shown in \Cref{fig: test-1 snapmat calib}, spatial discontinuities in the transformed solution move very little in the time domain. This induces a fast singular value decay in the matrix $\snapMat{G}{1}{K}$ and improves the approximability of a transformed solution in the POD basis. Similarly, the de-transformer $\td\varphi(\cdot,t)$ appears to be smooth in the space-time domain, which results in a fast singular value decay in $\snapMat{\td\Phi}{1}{K}$---see \Cref{fig: test-1 snapmat phi}. Also note that $\td\varphi(\cdot,t)$ does not exhibit any steep gradients, which, as noted in \Cref{lemma: err bound}, is desirable. 
 
 Observe that the singular values of $\snapMat{G}{1}{K}$ appear to stagnate. The error introduced by this stagnation is of little practical relevance because it is of the same order of magnitude as the error introduced by the HF solver. We refer to \cite{SarnaCalib2020} for further details.

\begin{figure}[ht!]
\centering
\subfloat[Comparison of singular value decay]{\includegraphics[width=2.5in]{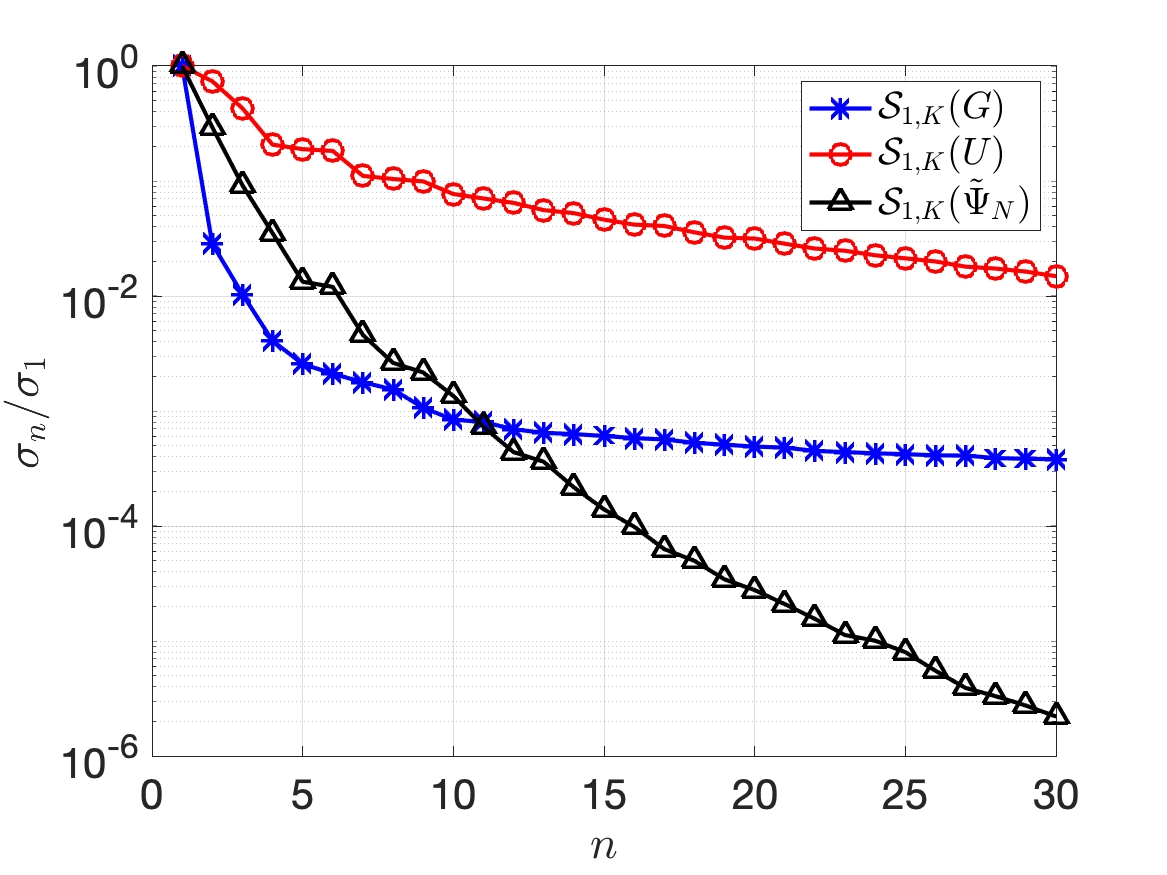}\label{fig: test-1 sigma}} 
\hfill
\subfloat[Solution $u_N(x,t)$]{\includegraphics[width=2.5in]{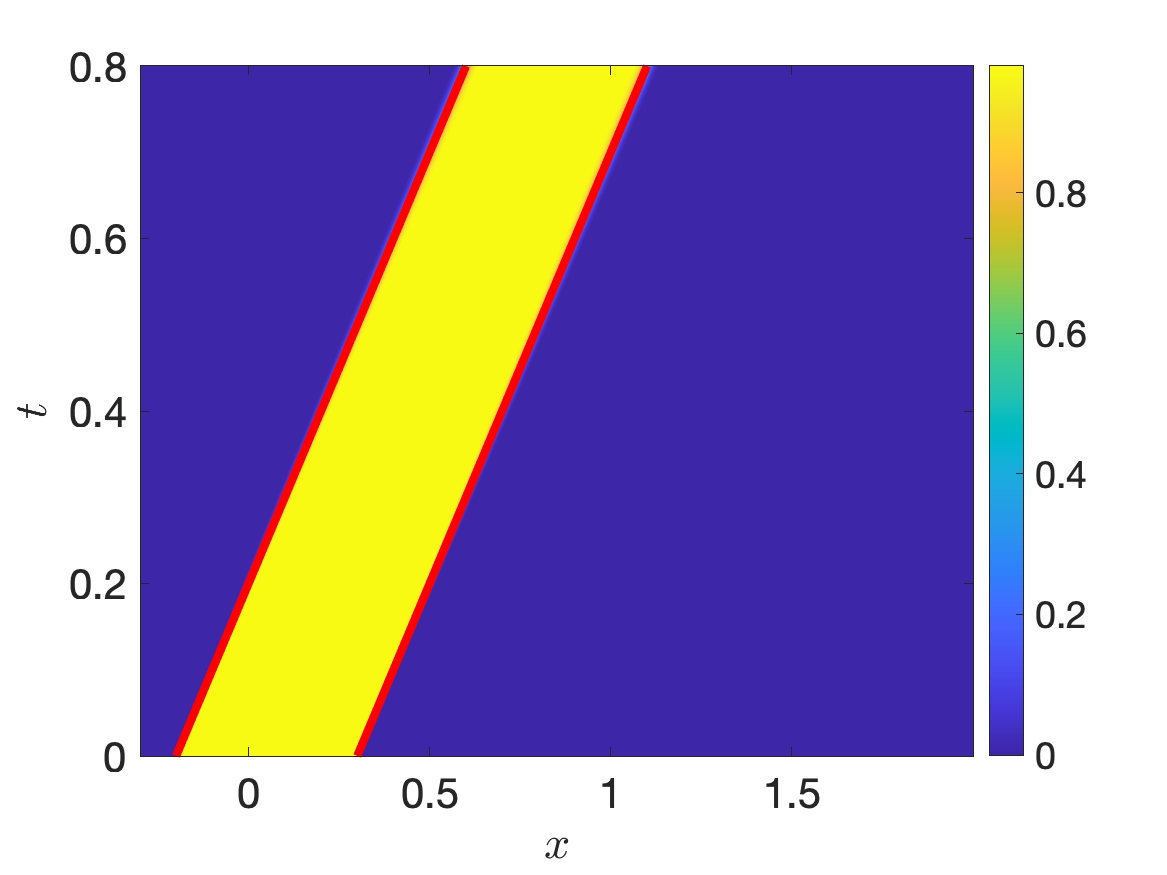}\label{fig: test-1 snapmat}}
\hfill
\subfloat[Transformed solution $g_N(x,t)$]{\includegraphics[width=2.5in]{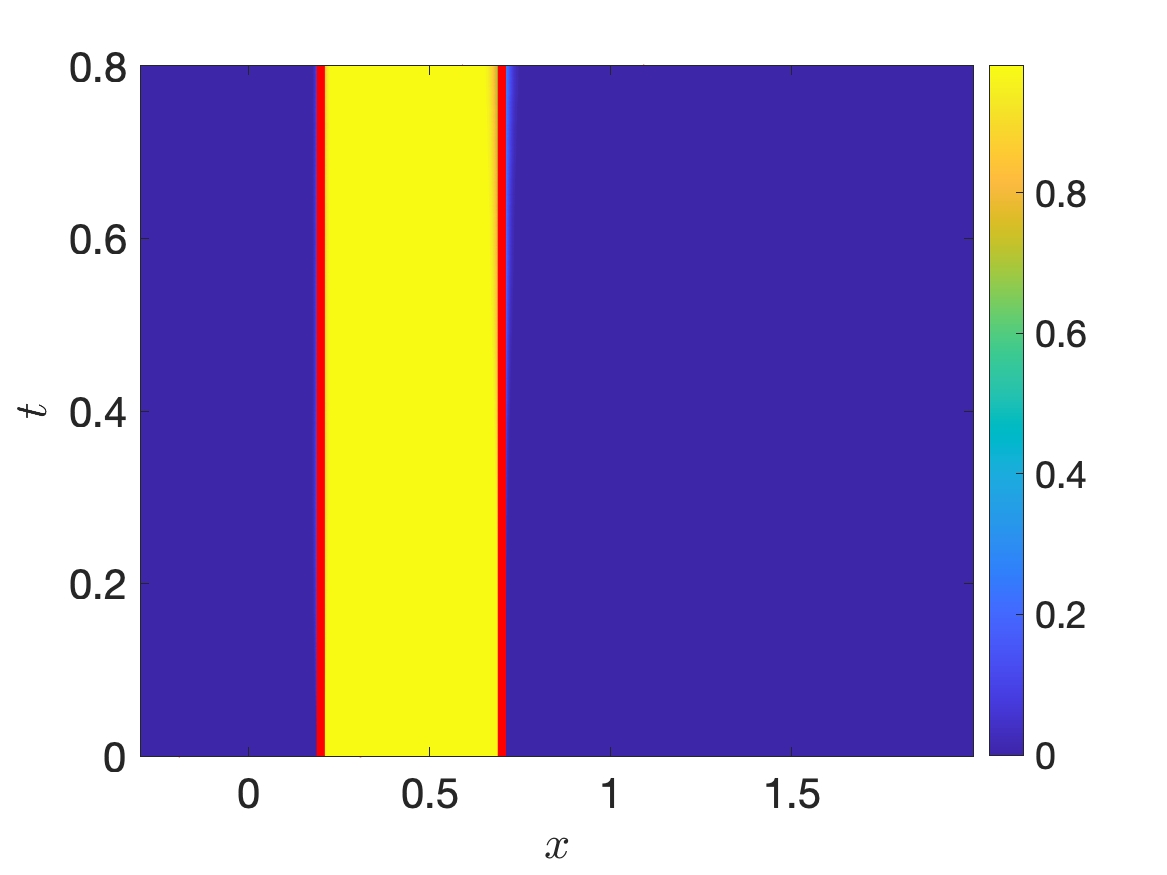}\label{fig: test-1 snapmat calib}}
\hfill
\subfloat[De-transformer $\td\varphi(\cdot,t)$]{\includegraphics[width=2.5in]{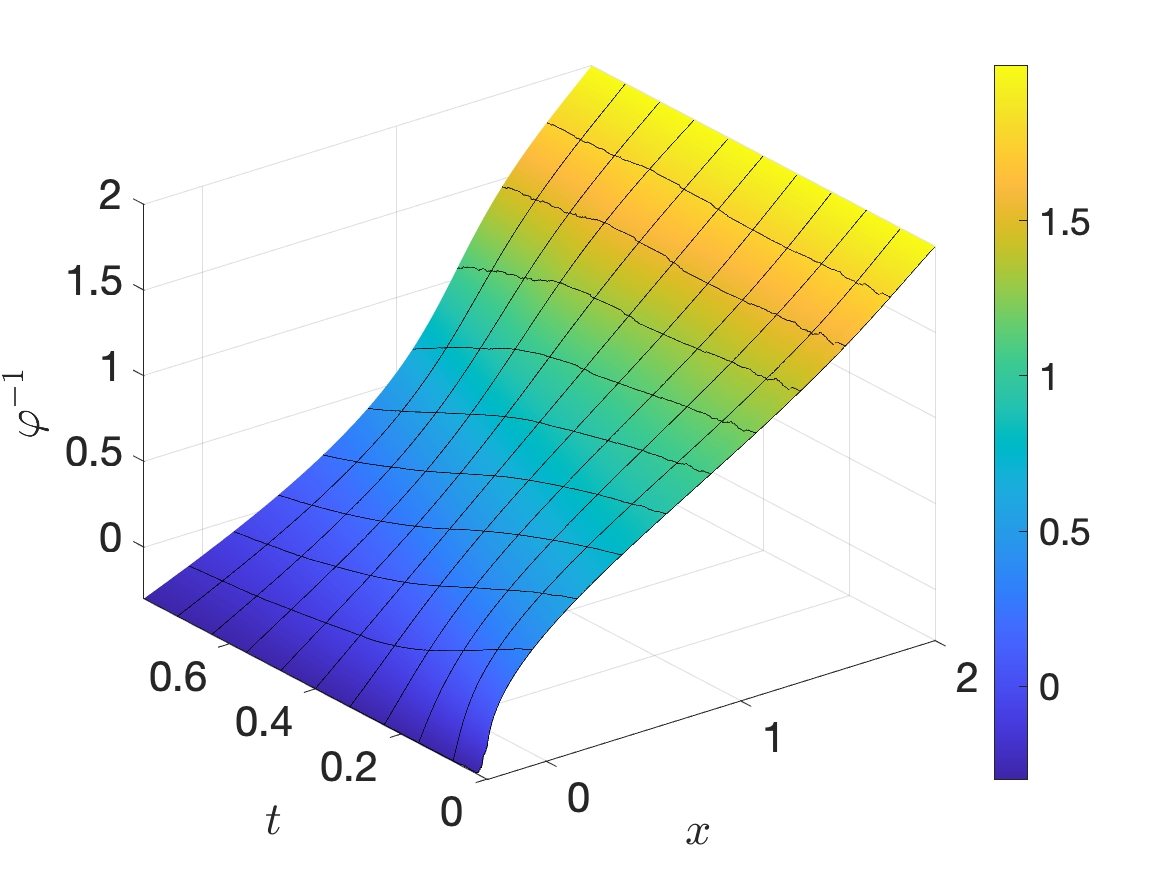}\label{fig: test-1 snapmat phi}}
\caption{Results for test--1. }	
\end{figure} 

\subsubsection{Approximation error} Following the discussion in \Cref{sec: standard DMD}, we study the approximation error separately for the interpolation $t\in [0,0.8]$ and extrapolation $t > 0.8$ regime. We first consider the interpolation regime. \Cref{fig: test-1 err intrp} compares the average error. TS--DMD outperforms DMD. The difference is the most pronounced for $n=11$, for which TS--DMD results in an error of $0.42\%$. Compare this to DMD which provides an error of $18.4\%$. For $n\geq 9$, the error from TS--DMD appears to stagnate because of the stagnation in the singular values reported earlier. 

We now consider the extrapolation regime $t > 0.8$---recall that we only collect snapshots inside $[0,0.8]$. \Cref{fig: test-1 err extrp} presents the average error. Error from DMD oscillates around $100\%$, and it appears that increasing the value of $n$ does not lead to convergence. TS--DMD performs much better. For all $n\geq 5$, it results in a maximum error of $9\%$, which is almost two orders of magnitude smaller than that resulting from DMD. 

Let us elaborate on the poor performance of DMD for the extrapolation regime---\Cref{sec: standard DMD} provides further elaboration. Ignoring the errors introduced by the HF solver, one can check that support of the solution $\solFV{t}$ is given via $K_t:=[-0.2 + t,0.3 +t]\subset \Omega$. Furthermore, the inner-product $\lan \solFV{t},\solFV{t^*}\ran_{L^2(\Omega)}$ reads $|K_t\cap K_{t^*}|$.  For $t > 0.8$, as the gap $t-0.8$ increases, the maximum overlap $\sup_{t^*\in [0,0.8]}|K_t\cap K_{t^*}|$ decreases. Consequently, in the $L^2$-sense, as time $t > 0.8$ increases, $\solFV{t}$ points more and more in a direction that is orthogonal to the previously collected snapshots. This decreases the approximability of $\solFV{t}$ in the span of snapshots collected from $[0,0.8]$ (and thus, also in the span of the POD modes) and thus, results in a poor DMD approximation. We emphasize that increasing the number of solution snapshots does not improve the accuracy. 

On the other hand, the support of the transformed solution $g_N(\cdot,t)$ is time-invariant---ignoring the slight discontinuity misalignment reported above---and is given by
$[-0.2 + \tref,0.3+\tref]$, where $\tref = 0.4$. As a result, even for $t > 0.8$, $g_N(\cdot,t)$ is well-approximable in the span of the transformed snapshots collected from $[0,0.8]$. Similar observation holds for the de-transformer $\td\varphi(\cdot,t)$. An accurate approximation of $g_N(\cdot,t)$ and $\td\varphi(\cdot,t)$, eventually, results in an accurate TS--DMD approximation of $u_N(\cdot,t)$. 

For $n=13$, time variation of the $L^1$ error is shown in \Cref{fig: test-1 err time}. At all times, the error from TS--DMD is almost an order of magnitude smaller than that from DMD. Observe that in the extrapolation regime---i.e., for $t > 0.8$---the error from both the methods grows monotonically with $t$; our findings here are consistent with that in \cite{Lu2020,lu2020prediction}. Nevertheless, TS--DMD provides a reasonable extrapolation with a maximum error of $16\%$. 

\begin{figure}[ht!]
\centering
\subfloat[Average error: Interpolation regime]{\includegraphics[width=2.5in]{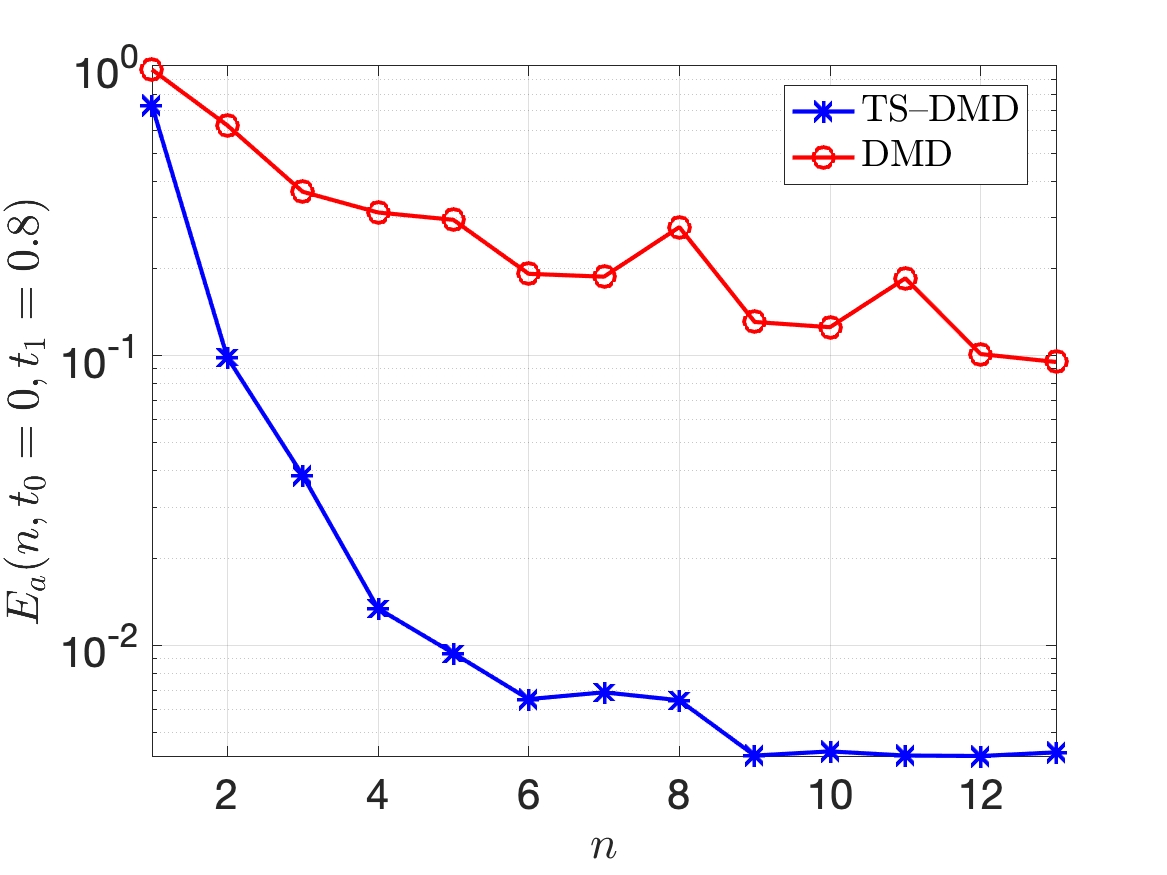}\label{fig: test-1 err intrp}} 
\hfill
\subfloat[Average error: Extrapolation regime]{\includegraphics[width=2.5in]{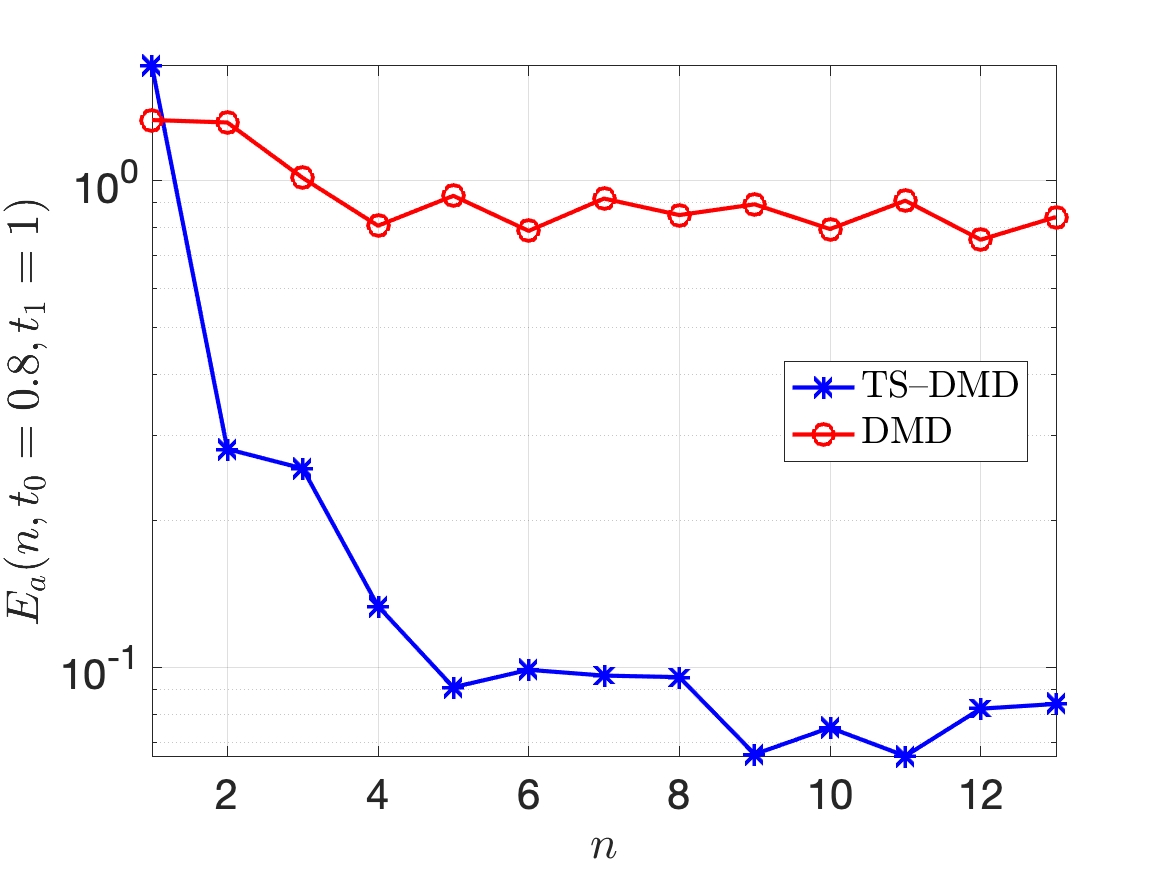}\label{fig: test-1 err extrp}} 
\caption{Results for test-1. }	
\end{figure} 

\begin{figure}[ht!]
\centering
\subfloat[Temporal behaviour of the $L^1$ error: $n=13$]{\includegraphics[width=2.5in]{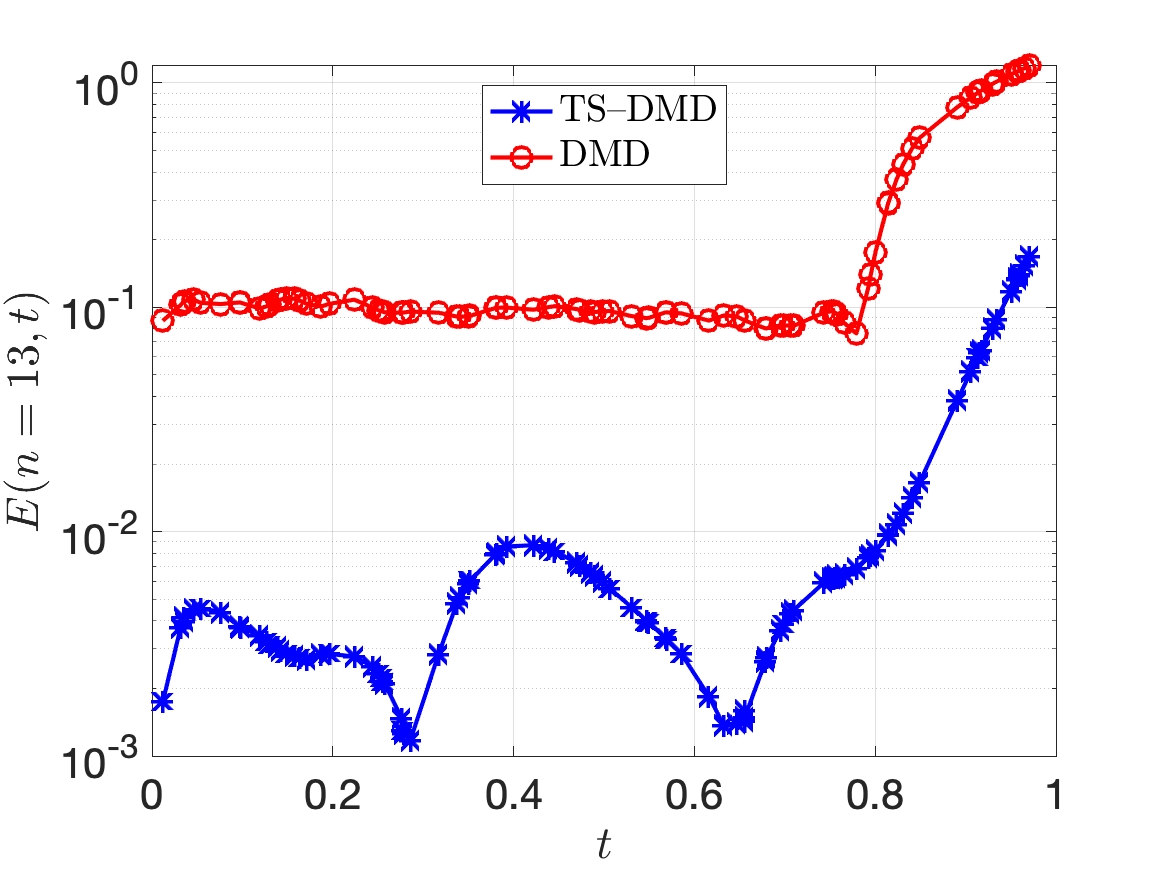}\label{fig: test-1 err time}} 
\hfill
\subfloat[Speed-up vs. error]{\includegraphics[width=2.5in]{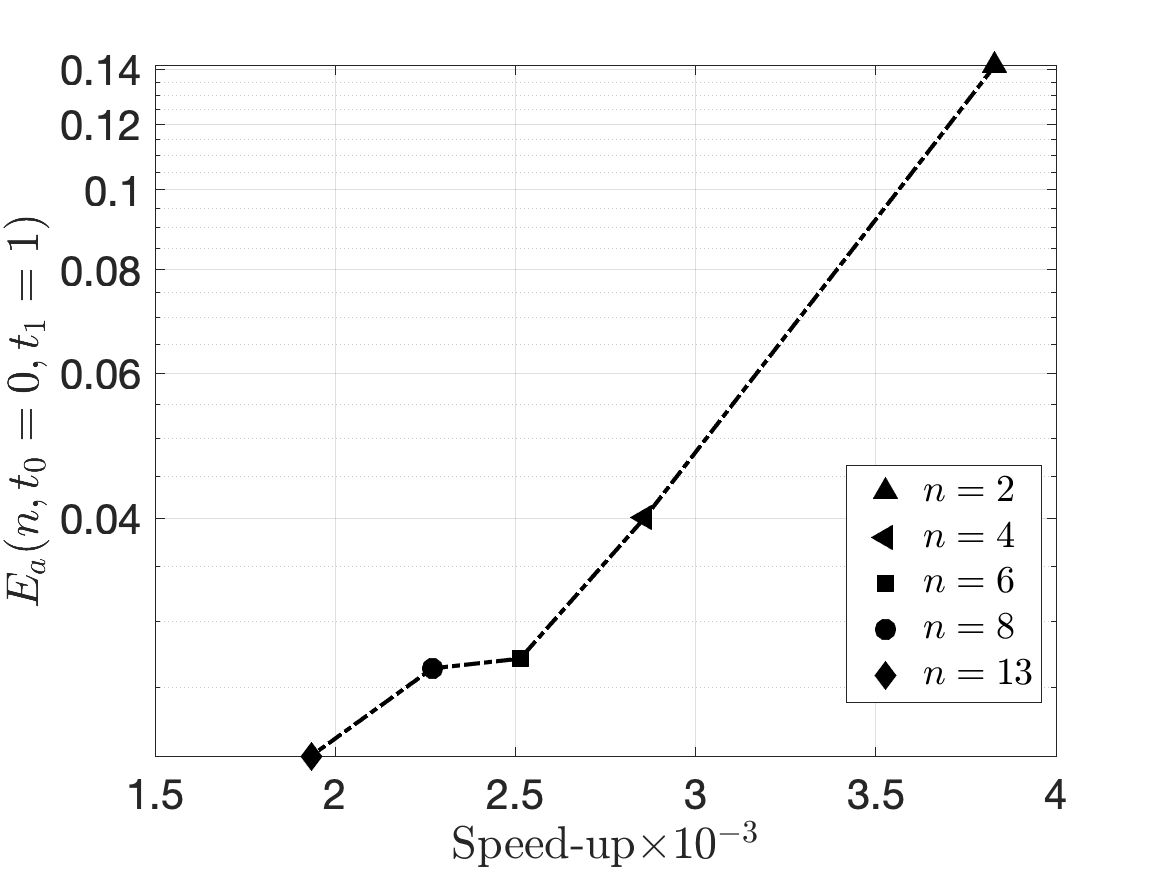}\label{fig: test-1 speedup}}
\caption{Results for test-1.}	
\end{figure} 

\subsubsection{Solution comparison}
A major drawback of the DMD approach is its highly oscillatory solution---see \Cref{fig: test-1 sol comp}. These oscillations are triggered by the temporal discontinuities in the solution and get stronger inside the extrapolation regime. In contrast, the TS--DMD solution exhibits no such oscillations. Note that oscillations in a ROM for advection-dominated problems are not just limited to the DMD approach. Indeed, as explained earlier, they are a shortcoming of using a linear reduced space and for that matter, also the works in \cite{Fresca2020,GNAT,Lu2020,MojganiLagrangian} report similar oscillations.

\begin{figure}[ht!]
\centering
\subfloat[TS-DMD solution: $n=13$]{\includegraphics[width=2.5in]{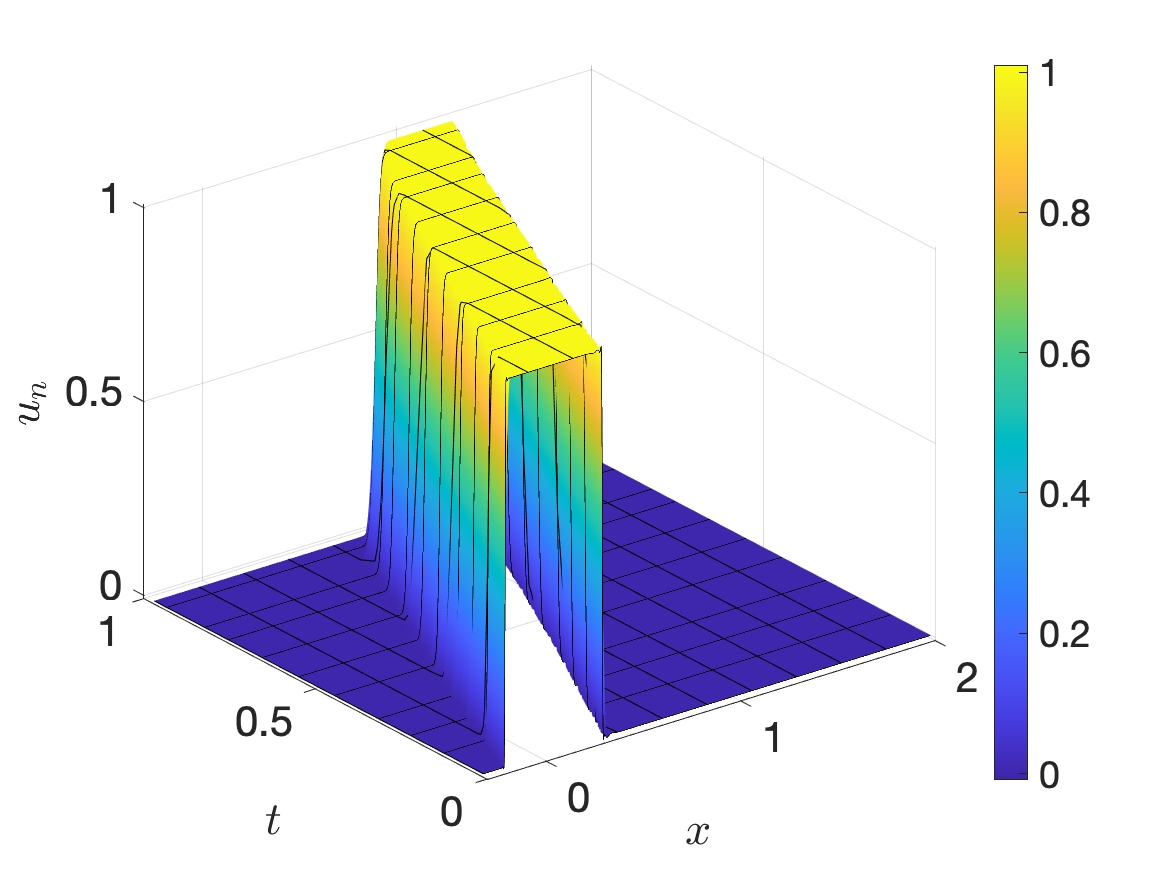}}
\hfill
\subfloat[DMD solution: $n=13$]{\includegraphics[width=2.5in]{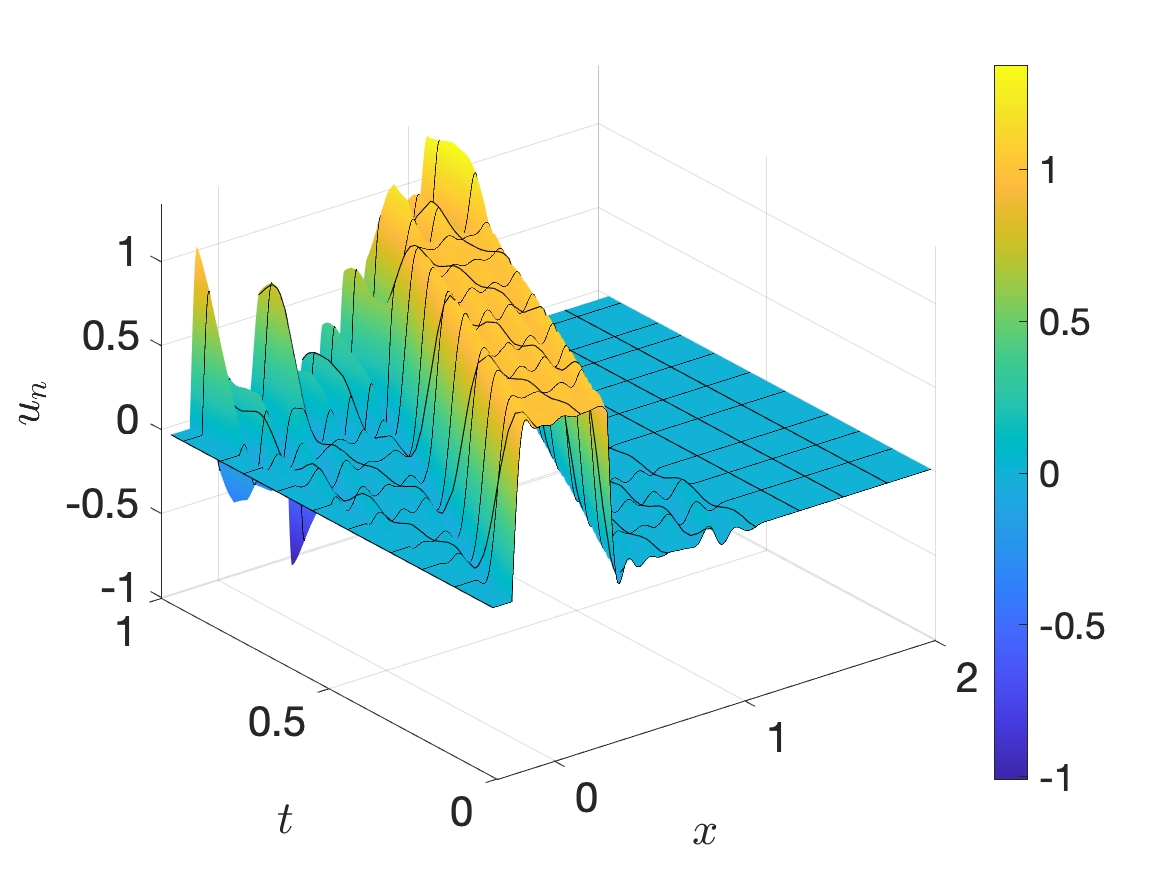}}
\hfill
\subfloat[HF solution]{\includegraphics[width=2.5in]{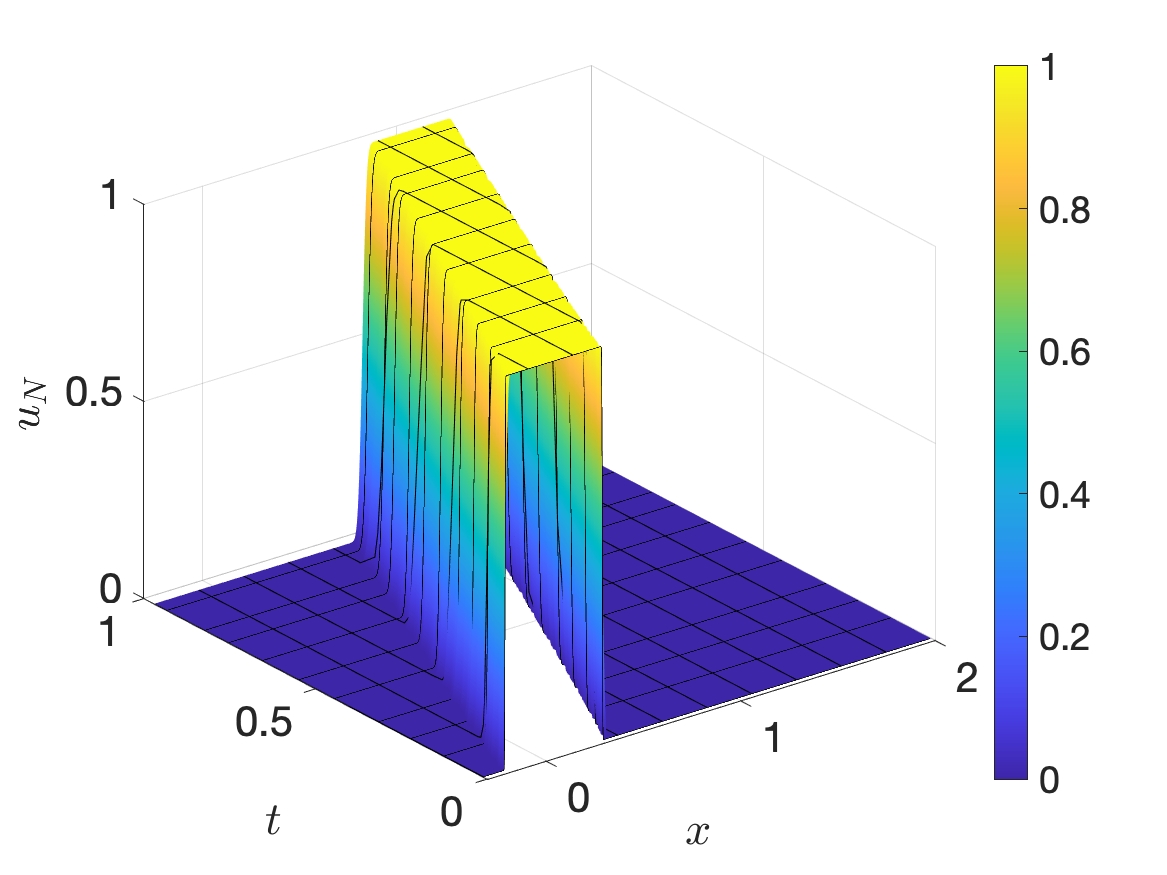}}
\caption{Results for test-1. \label{fig: test-1 sol comp}}	
\end{figure}

\subsubsection{Study of $\td\varphi_n$} In \Cref{lemma: err bound}, we assumed that the ROM $\td\varphi_n(\cdot,t)$ for $\td\varphi(\cdot,t)$ is a homeomorphism. Here, we justify this assumption empirically. We first define a Jacobian 
$$
\td{\mcal J}_n(x,t) := \pd_x \td\varphi_n(x,t),
$$
where we interpret the derivative in a weak sense.
For all $t\in [0,T]$, we want the Jacobian to be positive. Since, for all $t\in [0,T]$, $\td\varphi_n(\pd\Omega,t) = \pd\Omega$, this would then imply that $\td\varphi_n(\cdot,t)$ is a homeomorphism \cite{RegisterMOR}. We restrict our study to two limiting values of $n$: $n = 1$ and $n=10$. Qualitatively, the results remain similar for all the other intermediate values. We study the Jacobian at $100$ uniformly sampled points from $[0,T]$.

For $n=1$, \Cref{fig: test-1 phi1} plots $\td\varphi_n$ over the space-time domain.  With little deviations, $\td\varphi_n$ is almost the same as the identity mapping. Furthermore, the minimum value of the Jacobian $\inf_{x\in\Omega}\td{\mcal J}_n(x,t)$ stays well above zero, which is desirable---see \Cref{fig: test-1 jacob phi1}. Note that because the value of $n$ is small, $\td\varphi_n$ is a crude approximation of the true $\td\varphi(\cdot,t)$ shown in \Cref{fig: test-1 snapmat phi}. Nevertheless, increasing $n$ to $10$, improves the approximation quality---see \Cref{fig: test-1 phi10}.  Also, as shown in \Cref{fig: test-1 jacob phi10}, for $n=10$, the Jacobian is still positive. 

\begin{figure}[ht!]
\centering
\subfloat[$\td\varphi_n$ for $n=1$]{\includegraphics[width=2.5in]{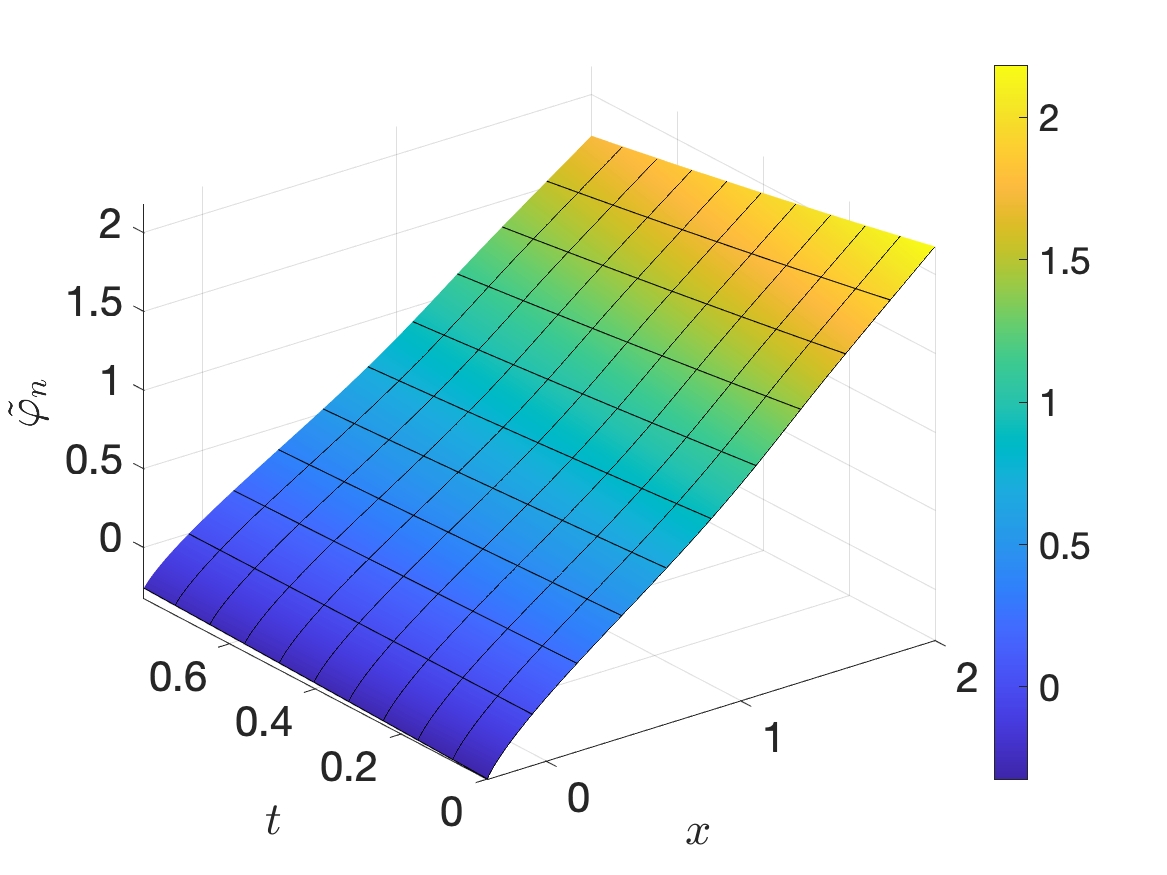}\label{fig: test-1 phi1}}
\hfill
\subfloat[Minimum value of the Jacobian for $n=1$]{\includegraphics[width=2.5in]{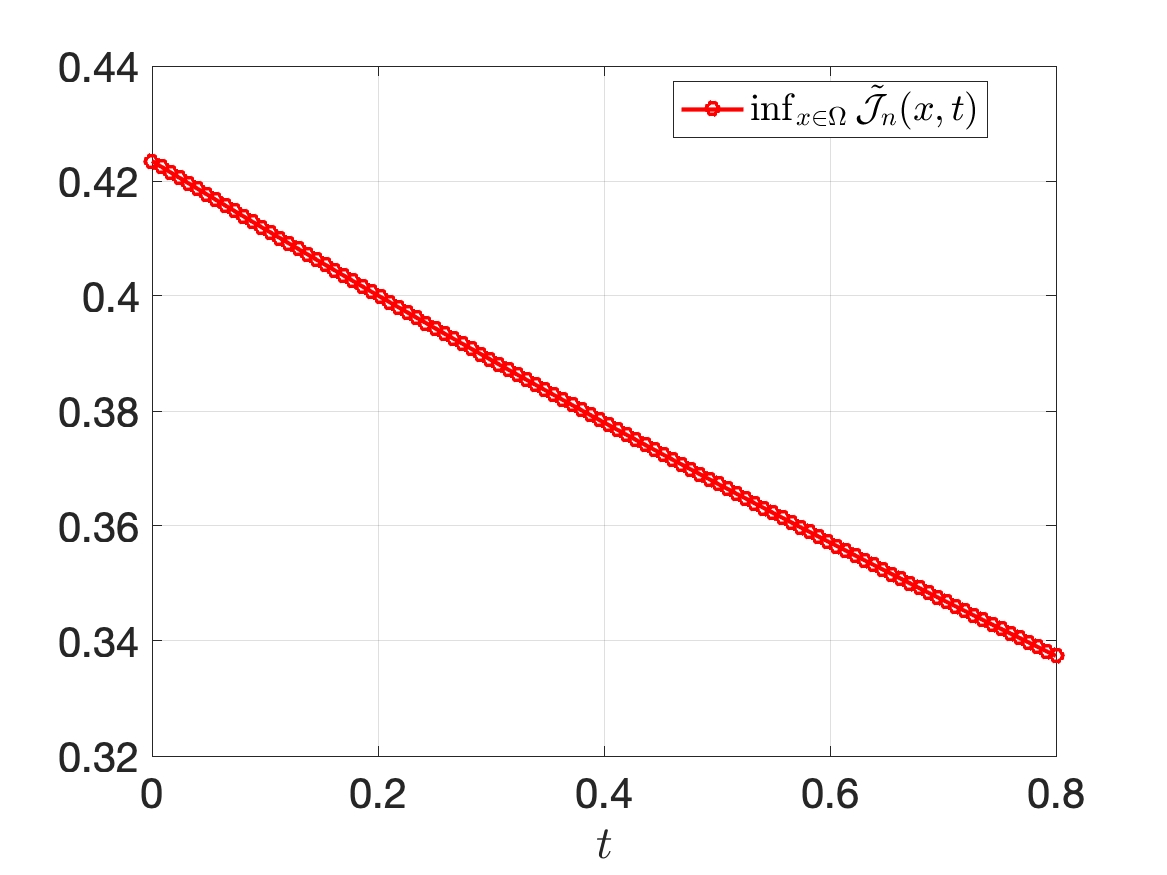}\label{fig: test-1 jacob phi1}}
\hfill
\subfloat[$\td\varphi_n$ for $n=10$]{\includegraphics[width=2.5in]{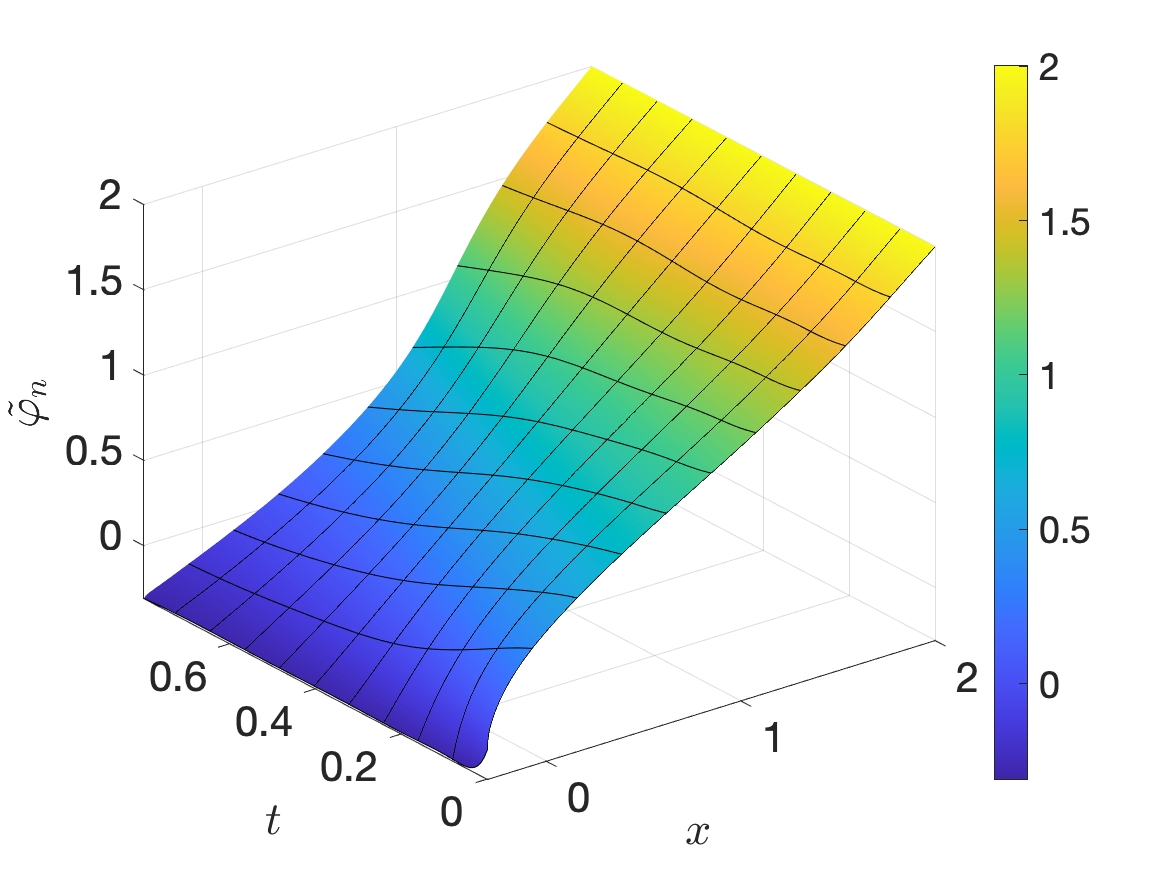}\label{fig: test-1 phi10}}
\hfill
\subfloat[Minimum value of the Jacobian for $n=10$]{\includegraphics[width=2.5in]{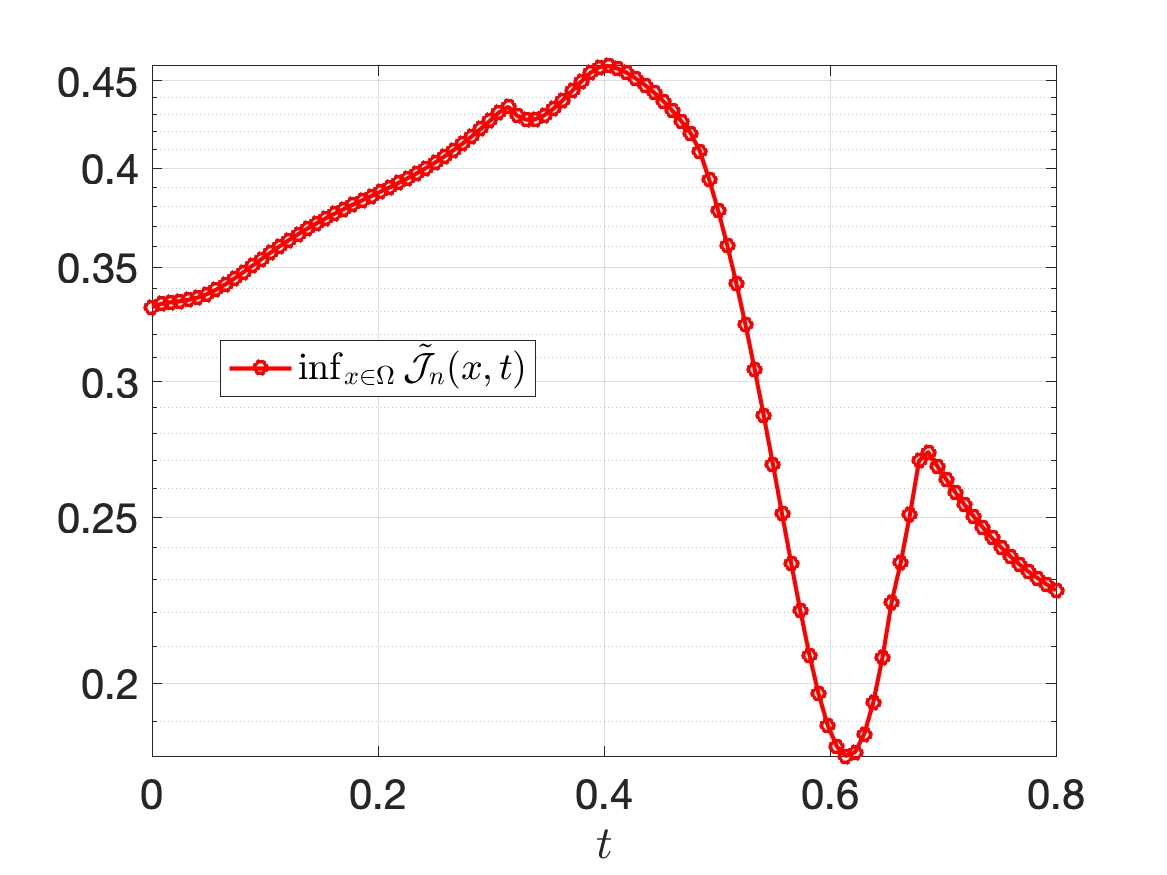}\label{fig: test-1 jacob phi10}}
\caption{Results for test-1.}	
\end{figure}

\subsubsection{Speed-up vs. accuracy} We study the speed-up offered by the TS--DMD approach. We define the speed-up as 
\begin{gather}
\kappa := \frac{\sum_{t\in \parSp_{test}}\tau_{HF}(t)}{\sum_{t\in\parSp_{test}}\tau_{ROM}(t)},\label{def kappa}
\end{gather}
where $\tau_{ROM}(t)$ and $\tau_{HF}(t)$ denote the CPU-time required by the TS--DMD and HF solver, respectively, to compute a solution at time $t\in [0,\infty)$. The set $\parSp_{test}$ contains $100$ uniformly and independently sampled points from $[0,1]$. The CPU-time is computed using the \texttt{tic-toc} function of MATLAB.

Recall that TS--DMD is an equation-free approach. To compute a solution at any $t\in [0,\infty)$, we just need to compute the sum in \eqref{rom DMD}. The HF solver, on the other hand, depending upon the value of $t$, requires some finite number of time steps to compute the solution. Therefore, we expect TS--DMD to be much more efficient than the HF solver. 

\Cref{fig: test-1 speedup} presents the results. For clarity, we plot only a handful of values. As anticipated, in general, both the error and speed-up decrease upon increasing $n$. Furthermore, the speed-up stays well above $10^3$. The minimum speed-up of $1.9\times 10^{3}$ results from $n=13$ and corresponds to an average error of $1.7\%$.

\subsection{Test-2} We discretize $\Omega$ with $4\times 10^3$ uniform elements, and collect snapshots at $500$ uniformly placed time instances inside $[0,0.6]$. We approximate $\dev{\tref}{t}$ in the polynomial space $\mcal P_{M=5}$, where the value of $M$ results from the procedure outlined in \Cref{sec: snapshots varphi}. We report the results for the first solution component $u_1$. Results for $u_2$ are similar and not discussed for brevity.

\subsubsection{Singular value decay} 
\Cref{fig: test-2 sigma} compares the singular value decay for the three snapshots matrix described earlier. Similar to the previous test case, because of the temporal discontinuities in the solution, singular values of the matrix $\snapMat{U}{1}{K}$ decay slowly. This is indicative of the poor approximability of the solution in a linear reduced space. Comparatively, the snapshot matrices $\snapMat{G}{1}{K}$ and $\snapMat{\td\Phi}{1}{K}$ exhibit a fast singular value decay. For instance, for $n=4$, we find
\begin{equation}
\begin{gathered}
\snapMat{U}{1}{K}:\hsp \frac{\sigma_n}{\sigma_1} = 14.7\times 10^{-2},\hspB \snapMat{G}{1}{K}:\hsp \frac{\sigma_n}{\sigma_1} = 0.3\times 10^{-2},\\
\snapMat{\td\Phi}{1}{K}:\hsp \frac{\sigma_n}{\sigma_1} = 1.3\times 10^{-2}.
\end{gathered}
\end{equation}
Our results indicate that it is better to first approximate the transformed solution and the de-transformer in the POD basis, followed by a re-composition based approximation for the solution. The average approximation errors reported below will further corroborate our claim.
\begin{figure}[ht!]
\centering
\subfloat[]{\includegraphics[width=2.5in]{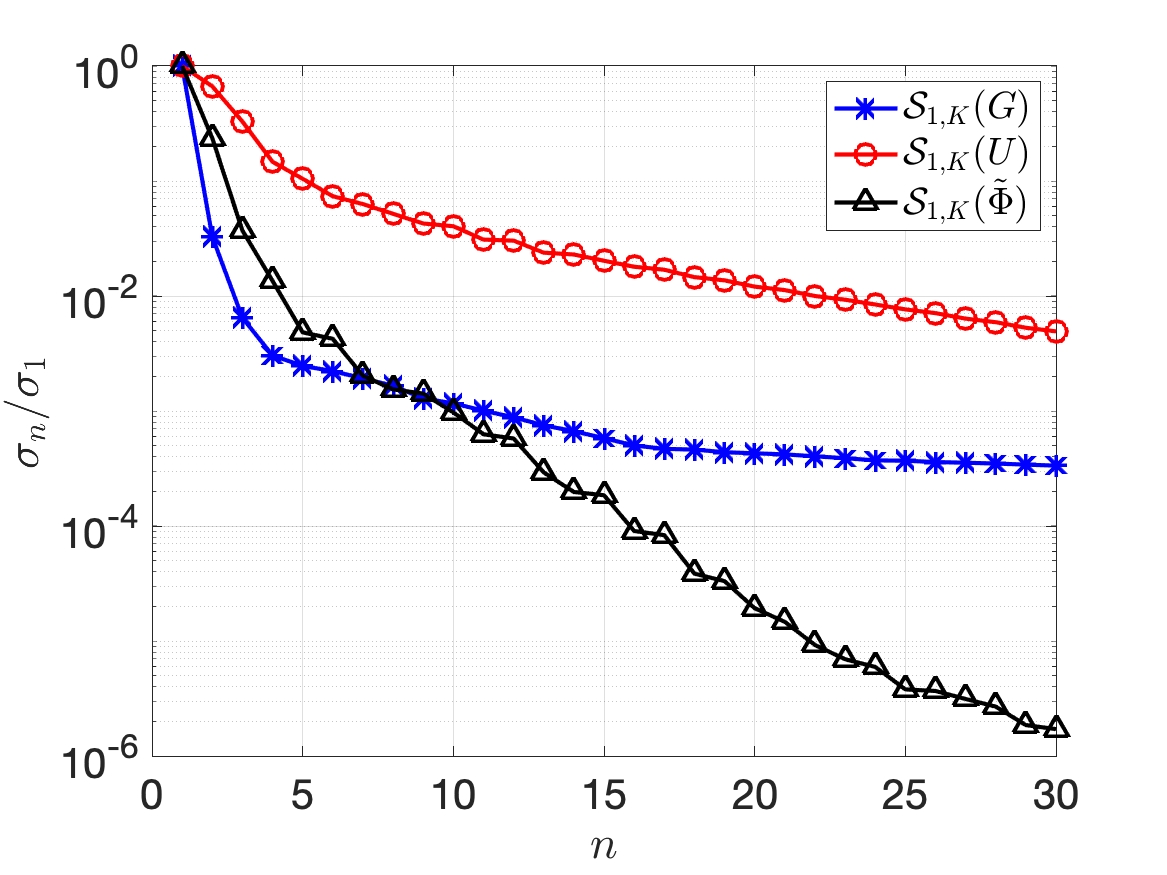}\label{fig: test-2 sigma}} 
\caption{Results for test-2. Comparison of singular value decay.}	
\end{figure} 

\subsubsection{Approximation error} 
\Cref{fig: test-2 err intrp} presents the average error in the interpolation regime $t\in [0,0.6]$. Our observations remain similar to the previous test case. The TS--DMD technique outperforms DMD. For $n=10$, it results in an error of $0.35\%$, which is almost fifteen times smaller than the error of $5.2\%$ resulting from the DMD approach. Notice that because of the stagnation in the singular value decay reported earlier, the error from TS-DMD stagnates for $n \gtrsim 6$. 

As one might anticipate from our earlier discussion, DMD performs poorly in the extrapolation regime $t\in [0.6,0.7]$---see \Cref{fig: test-2 err extrp}. For $n=10$, it results in an error of $\approx 27\%$, which is almost ten times larger than the error of $2.6\%$ resulting from TS--DMD. Same as the interpolation regime, the error from TS--DMD almost stagnates for $n\gtrsim 7$. 

For $n=10$, time variation of the relative $L^1$ error is shown in \Cref{fig: test-2 err time}. For both the methods, in the extrapolation regime, the error increases monotonically with time. Furthermore, as compared to DMD, at all time instances, TS--DMD provides an error that is at least an order of magnitude smaller. At the final time $t=0.7$, both the methods result in the maximum error with the error being $4.4\%$ and $37.1\%$ for TS--DMD and DMD, respectively.

\begin{figure}[ht!]
\centering
\subfloat[Average error: interpolation regime]{\includegraphics[width=2.5in]{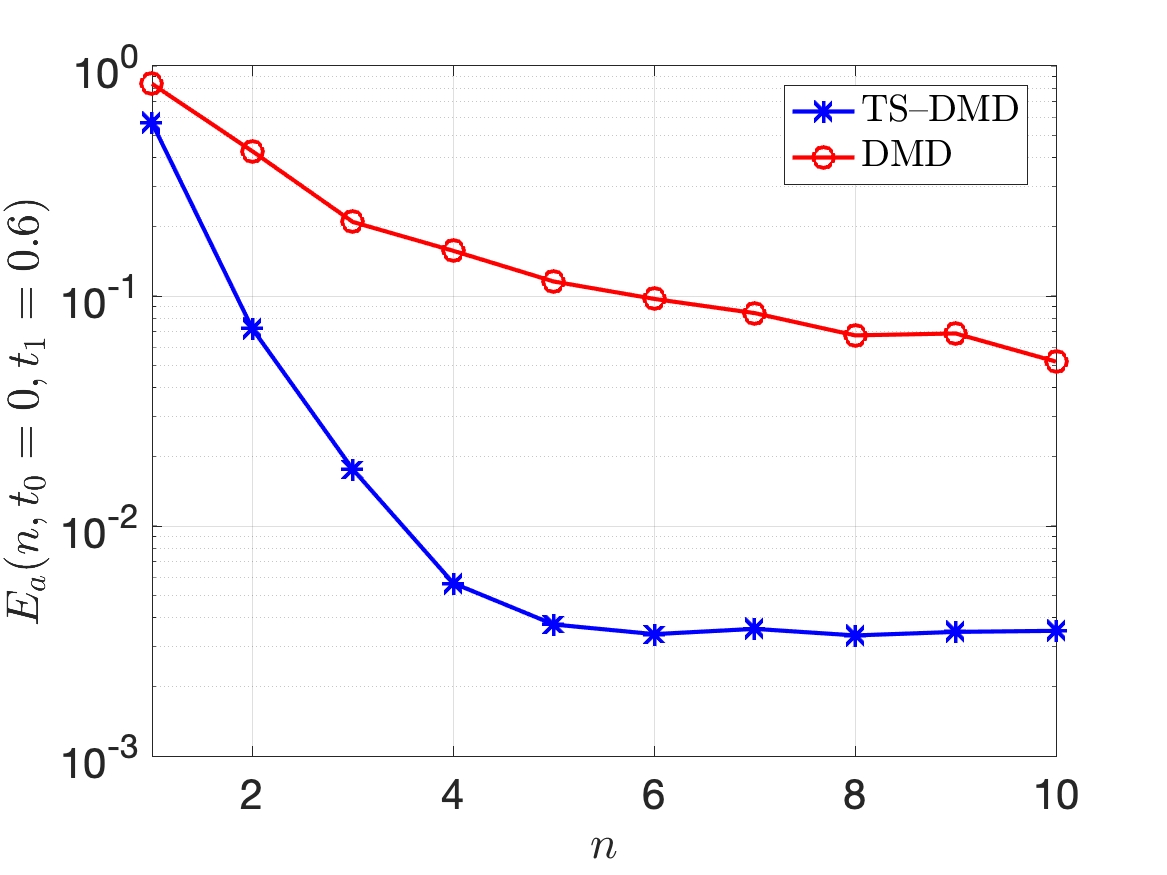}\label{fig: test-2 err intrp}} 
\hfill
\subfloat[Average error: extrapolation regime]{\includegraphics[width=2.5in]{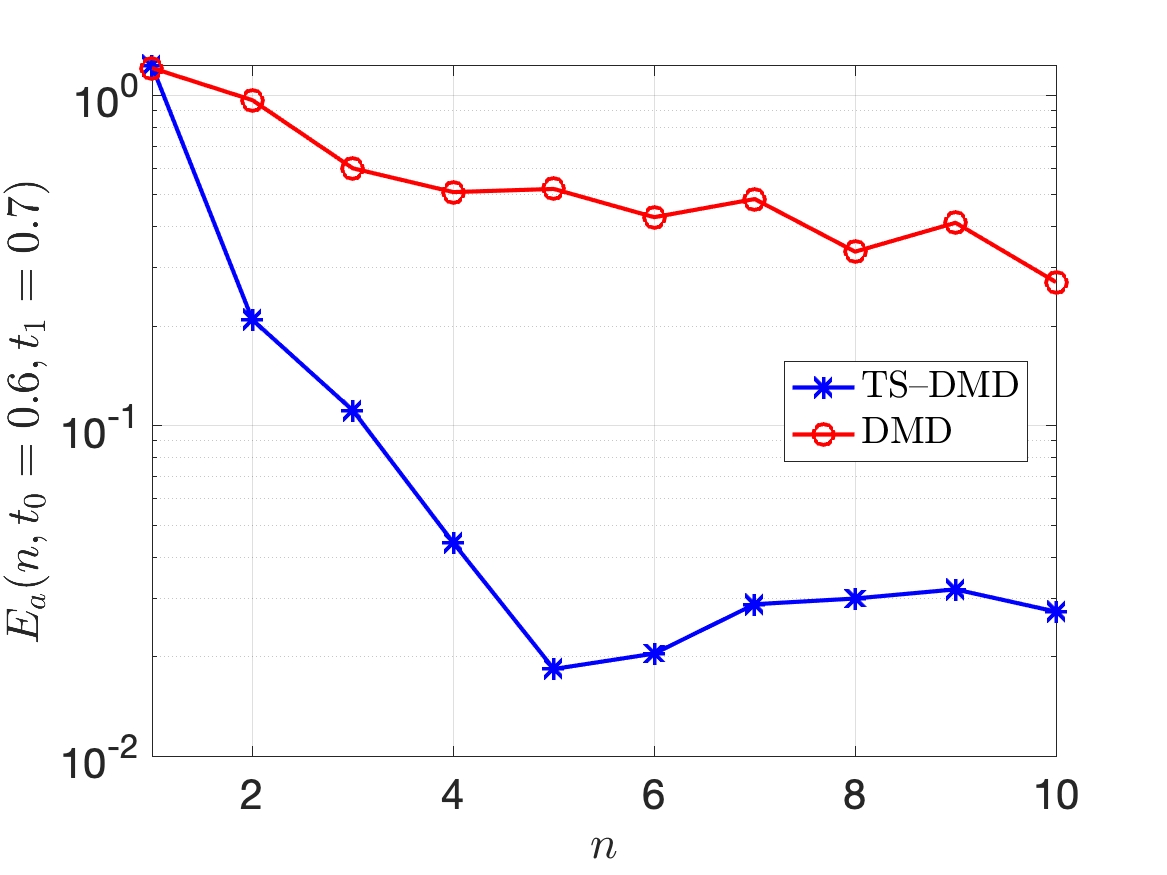}\label{fig: test-2 err extrp}} 
\caption{Results for test-2. }	
\end{figure} 

\begin{figure}[ht!]
\centering
\subfloat[Temporal behaviour of the $L^1$ error: $n=10$]{\includegraphics[width=2.5in]{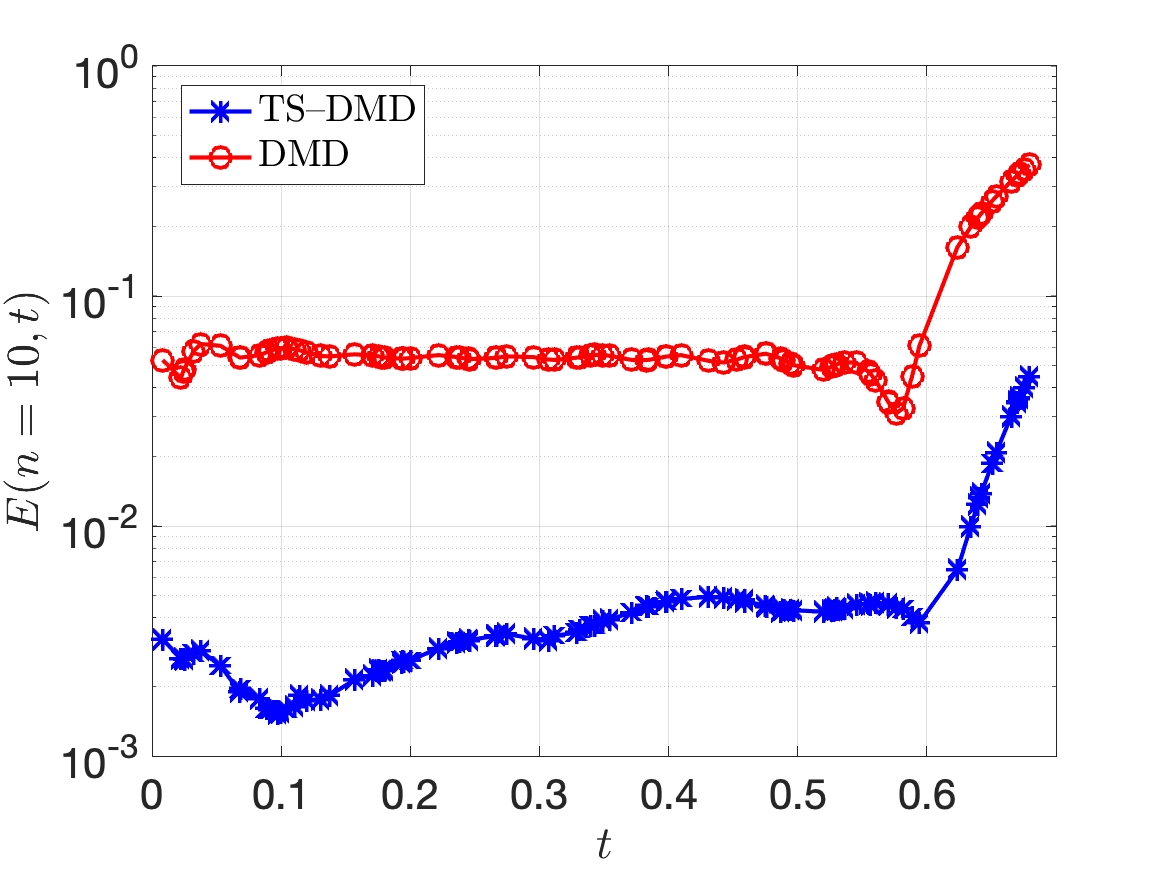}\label{fig: test-2 err time}}
\hfill
\subfloat[Speed-up vs. accuracy]{\includegraphics[width=2.5in]{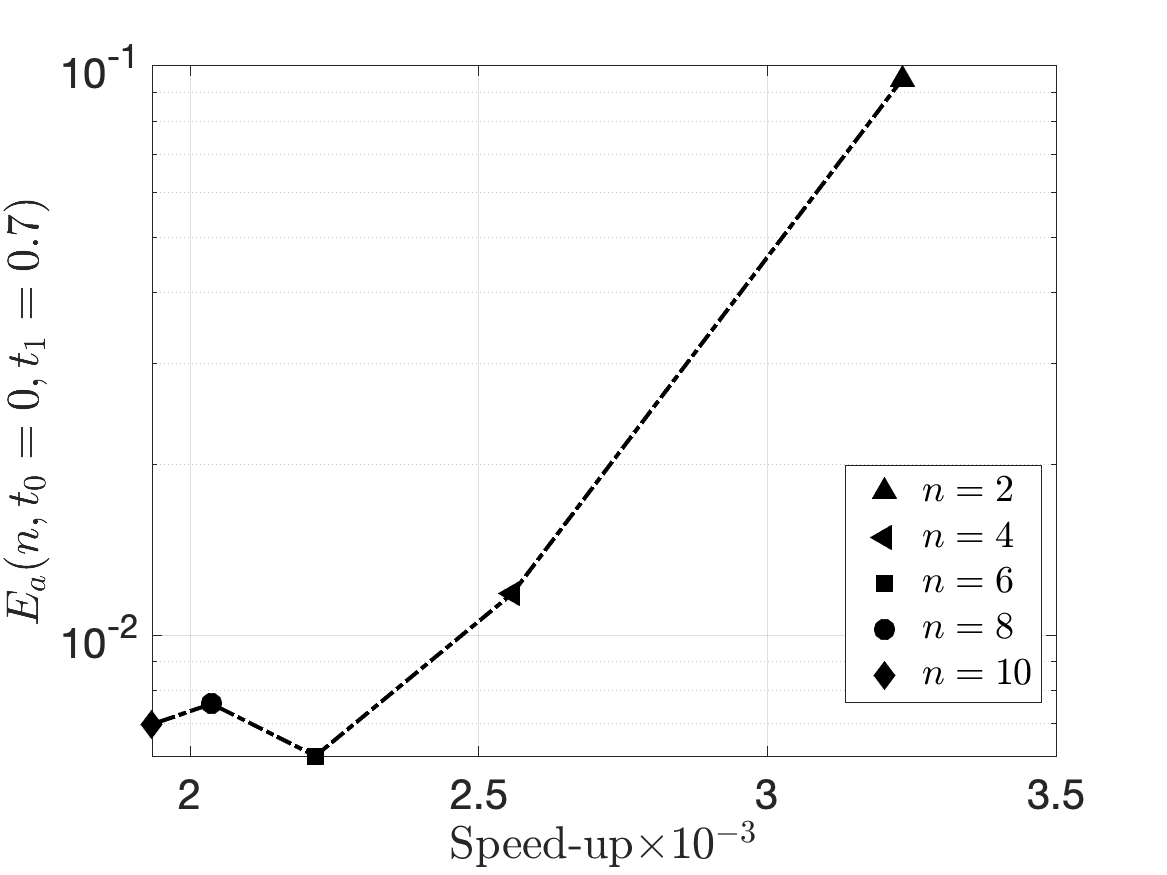}\label{fig: test-2 speedup}}
\caption{Results for test-2.}
\end{figure} 

\subsubsection{Solution comparison}
For $n=5$, \Cref{fig: test-2 sol comp} compares the different solutions. As expected, DMD provides an oscillatory solution. These oscillations are particularly strong in the extrapolation regime $t > 0.6$. Notice that in the extrapolation regime, DMD largely misrepresents the solution. It gets both the shape of the sin bump and the location of discontinuities wrong. Furthermore, it exhibits oscillations outside of the support of the HF solution. In contrast, TS--DMD accurately approximates the solution both in the interpolation and the extrapolation regime. In the \textit{eye-ball} norm, it is almost indistinguishable from the HF solution. Upon a closer inspection, next to the spatial discontinuities, one can find minor over and undershoots in the TS--DMD solution. These artifacts are a result of the minor misalignment in discontinuities reported earlier. 

\begin{figure}[ht!]
\centering
\subfloat[TS-DMD solution: $n=5$]{\includegraphics[width=2.5in]{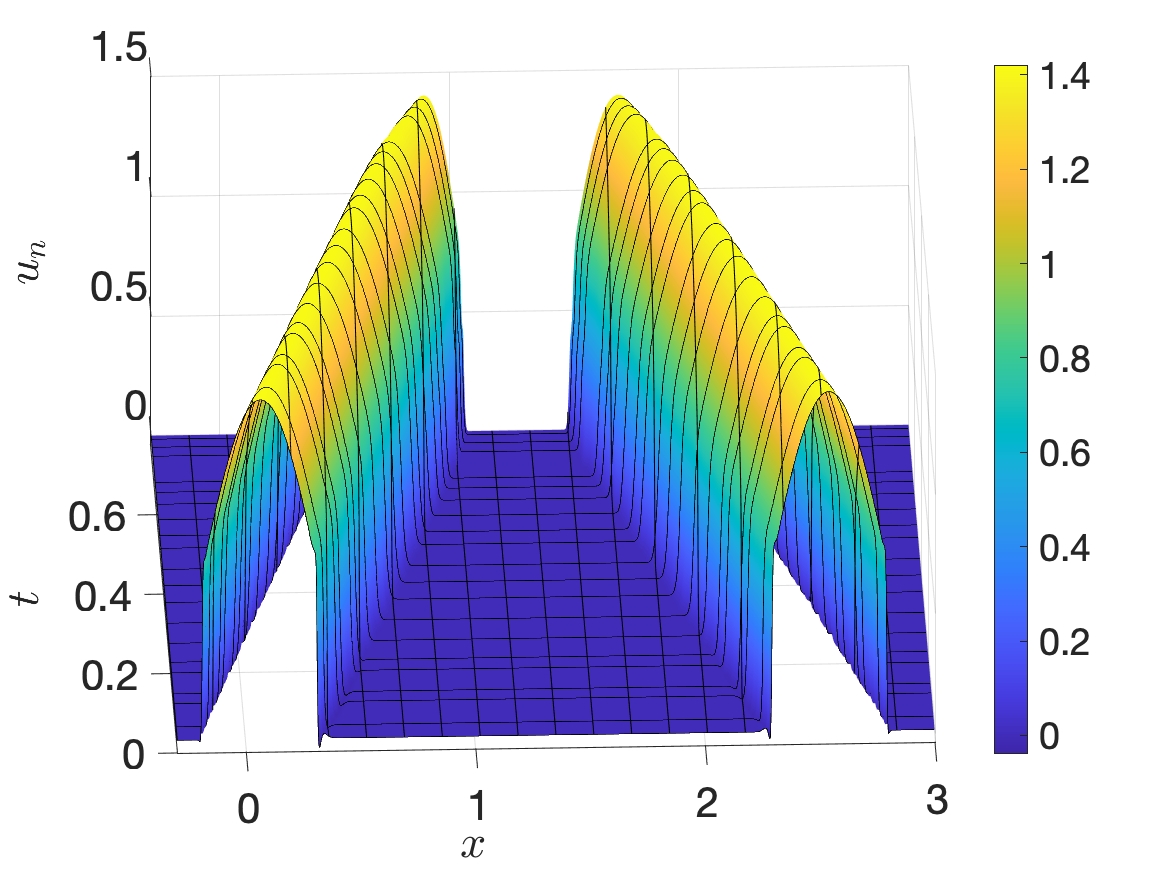}}
\hfill
\subfloat[DMD solution: $n=5$]{\includegraphics[width=2.5in]{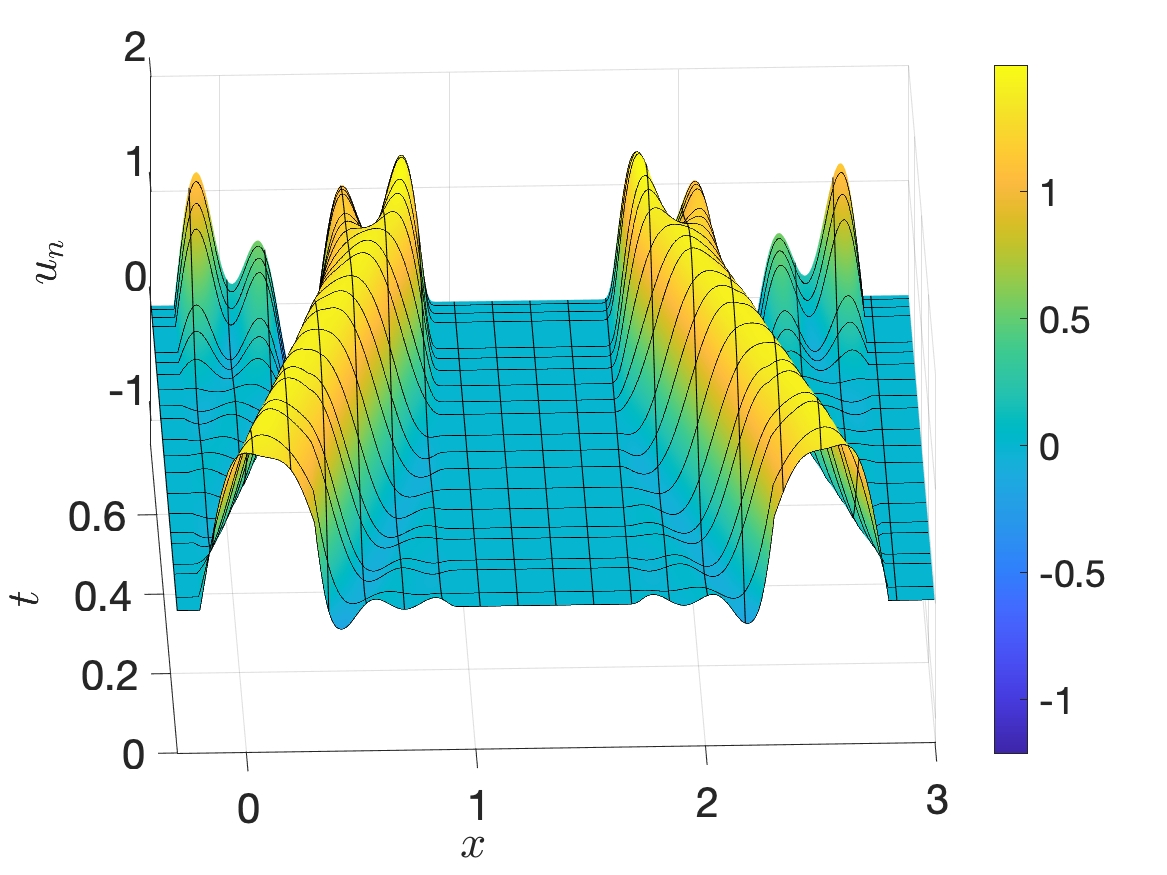}}
\hfill
\subfloat[HF solution]{\includegraphics[width=2.5in]{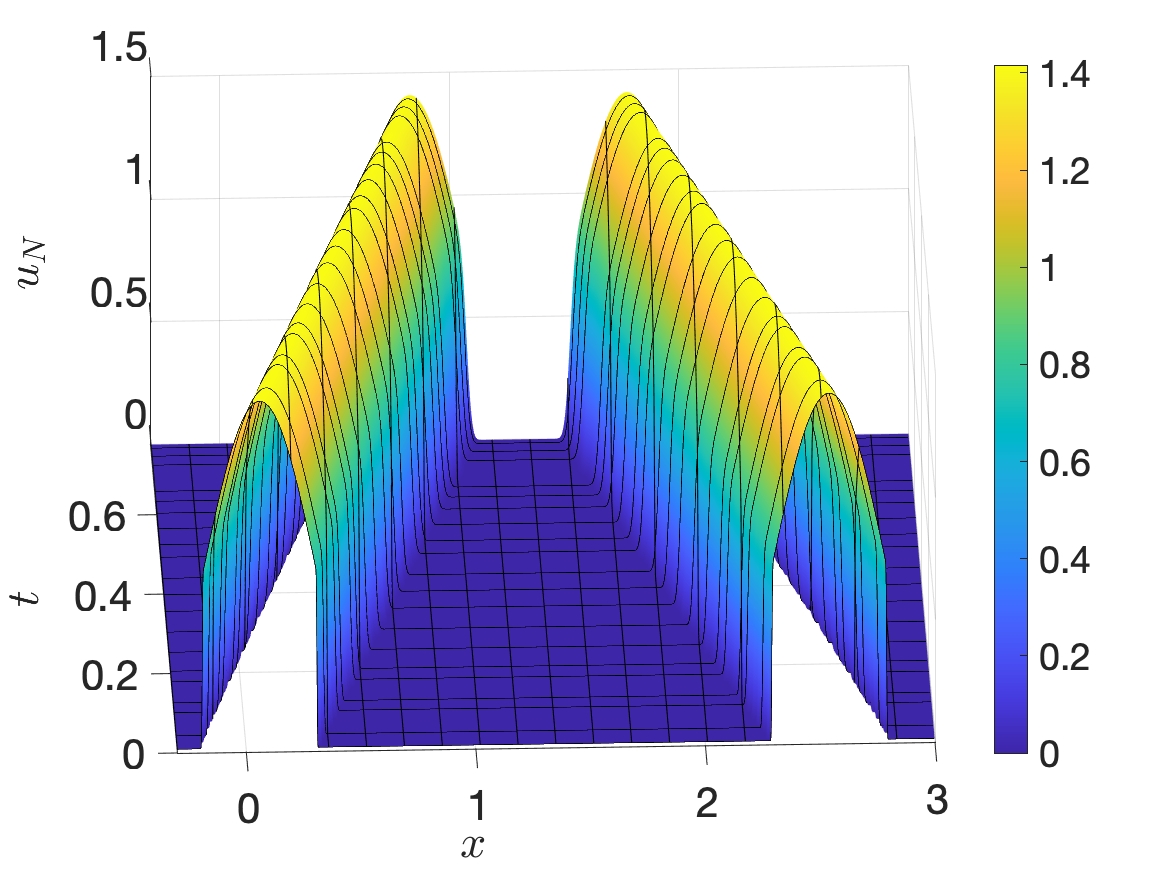}}
\caption{Results for test-1. \label{fig: test-2 sol comp}}	
\end{figure}

\subsubsection{Speed-up vs. accuracy} 
\Cref{fig: test-2 speedup} compares the speed-up defined in \eqref{def kappa} to the average error $E_a(n,t_0,t_1)$. Our findings remain similar to the previous test case. In general, both the error and the speed-up decrease upon increasing $n$. For $n=10$, we recorded a minimum speed-up of $1.9\times 10^{3}$ with a corresponding average error of $0.6\%$.

\subsection{Test-3}
We discretize $\Omega$ with $300\times 300$ grid cells, and collect snapshots at $50$ uniformly placed temporal points inside $[0,0.8]$. Note that we collect much fewer snapshots than in the previous two test cases. For multi-dimensional problems, snapshot collection is usually expensive and we demonstrate that our technique works well with fewer snapshots. We approximate the displacement field $\dev{1}{t}$ in $\mcal P_{M=6}$.
\subsubsection{Singular value decay}
\Cref{fig: test-3 singular value} compares the singular value decay of the three matrices: $\snapMat{U}{1}{K}$, $\snapMat{G}{1}{K}$ and $\snapMat{\td\Phi}{1}{K}$. Our observations are similar to the previous test cases. Singular values of the matrix $\snapMat{U}{1}{K}$ decay slowly, indicating that the solution has poor approximability in the POD basis. Comparatively, the singular values of $\snapMat{G}{1}{K}$ and $\snapMat{\td\Phi}{1}{K}$ decay faster. For instance, consider $n=4$, for which we find 
\begin{equation}
\begin{gathered}
\snapMat{U}{1}{K}:\hsp \frac{\sigma_n}{\sigma_1} = 18.8\times 10^{-2},\hspB \snapMat{G}{1}{K}:\hsp \frac{\sigma_n}{\sigma_1} = 4.7\times 10^{-2},\\
\snapMat{\td\Phi}{1}{K}:\hsp \frac{\sigma_n}{\sigma_1} = 3.5\times 10^{-3}.
\end{gathered}
\end{equation}
Notice that the singular value of $\snapMat{G}{1}{K}$ and $\snapMat{\td\Phi}{1}{K}$ is almost four and seventy times smaller, respectively, to that of $\snapMat{U}{1}{K}$. 
\begin{figure}[ht!]
\centering
\includegraphics[width=2.5in]{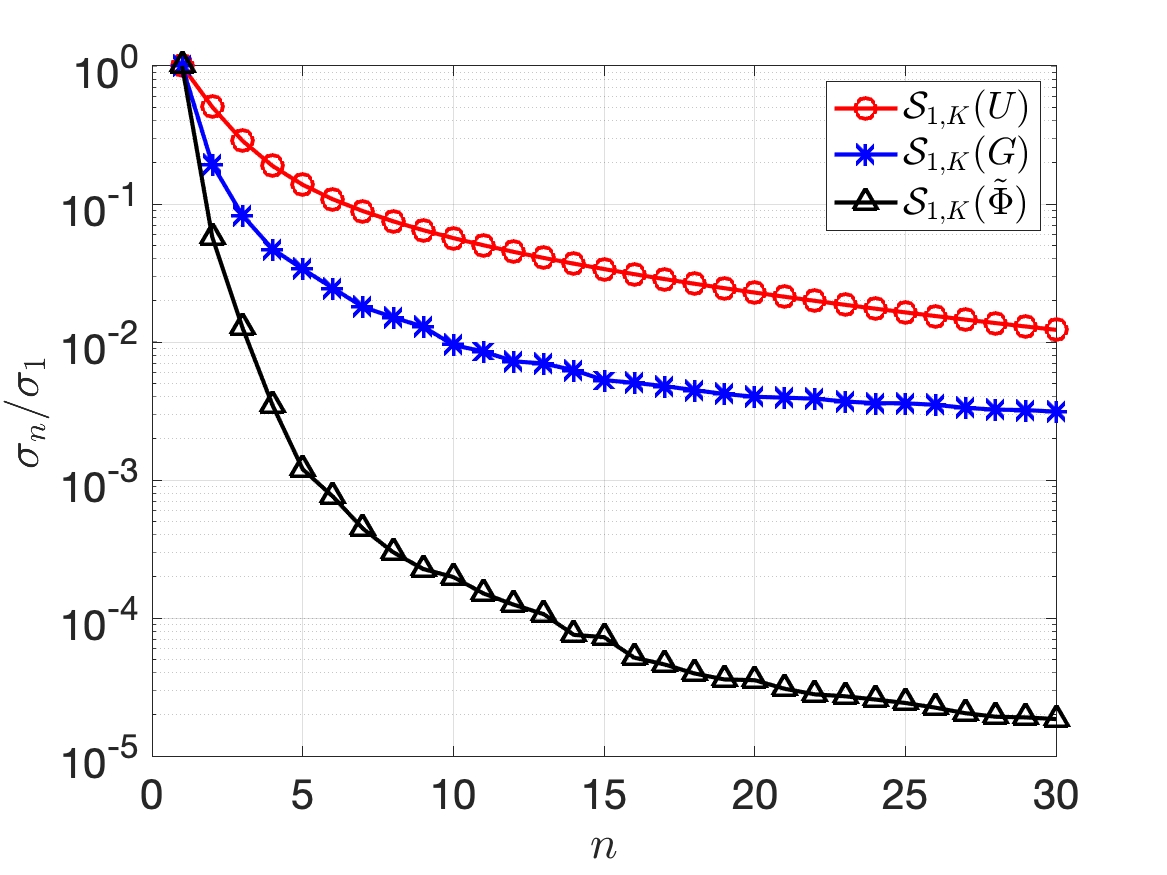}
\caption{Results for test-3. Comparison of singular value decay. \label{fig: test-3 singular value}}	
\end{figure} 

\subsubsection{Approximation error, solution comparison and speed-up} Results for the convergence study remain similar to the previous test cases and are not repeated again for brevity. Rather, we set $n=5$ and study the approximation error. First, we consider the interpolation regime $t\in [0,1]$, for which we find the following average errors
\begin{equation}
\begin{aligned}
\text{Interpolation regime} \begin{cases}
\text{TS--DMD}:\hspB & E_a(n=5,0,1)=1.7\times 10^{-2},\\
\text{DMD}:\hspB & E_a(n=5,0,1)=14.9\times 10^{-2}.
\end{cases}
\end{aligned}
\end{equation}
Observe that the error from TS--DMD is almost nine times smaller than that from DMD. \Cref{fig: test-3 sol comp} compares the solution at $t=1$. Same as earlier, the DMD solution exhibits strong oscillations and provides a poor approximation. In contrast, the TS--DMD solution is accurate and (almost) oscillation-free. Let us recall that $n=5$ is just $5.5\times 10^{-3}$ percent of the HF dimension $N$. For such a small value of $n$, a relative error of $1.7\%$ can be considered reasonable.
\begin{figure}[ht!]
\centering
\subfloat[Interpolation regime $t=1$: TS--DMD]{\includegraphics[width=2.5in]{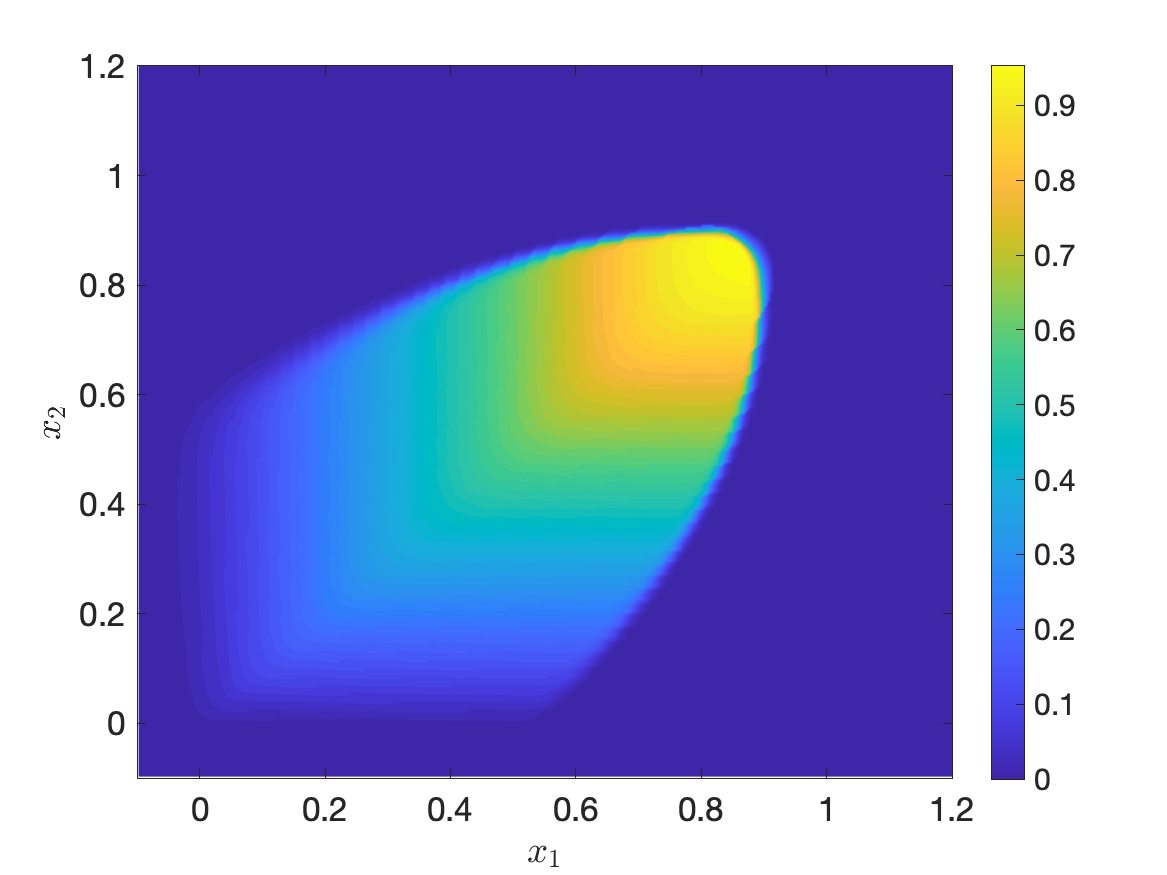}} 
\hfill
\subfloat[Interpolation regime $t=1$: DMD]{\includegraphics[width=2.5in]{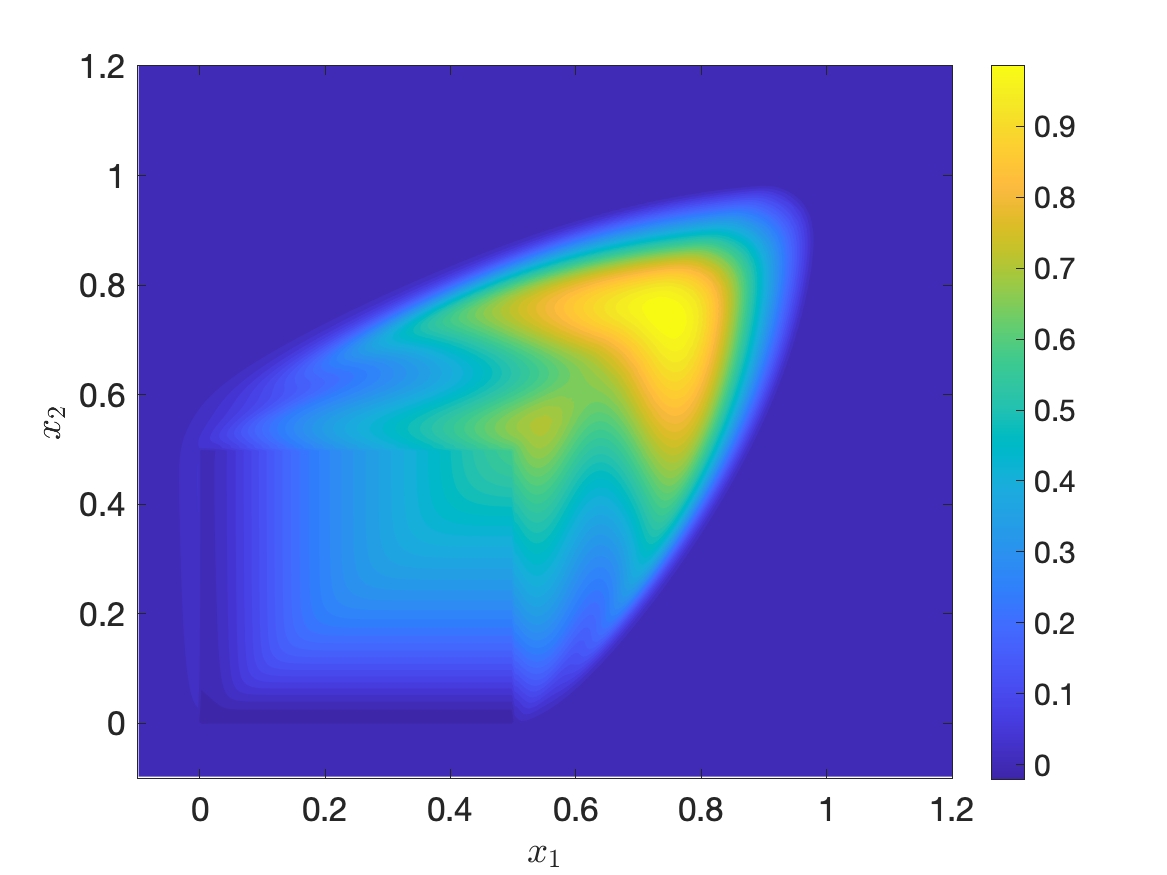}}
\hfill
\subfloat[Interpolation regime $t=1$: HF]{\includegraphics[width=2.5in]{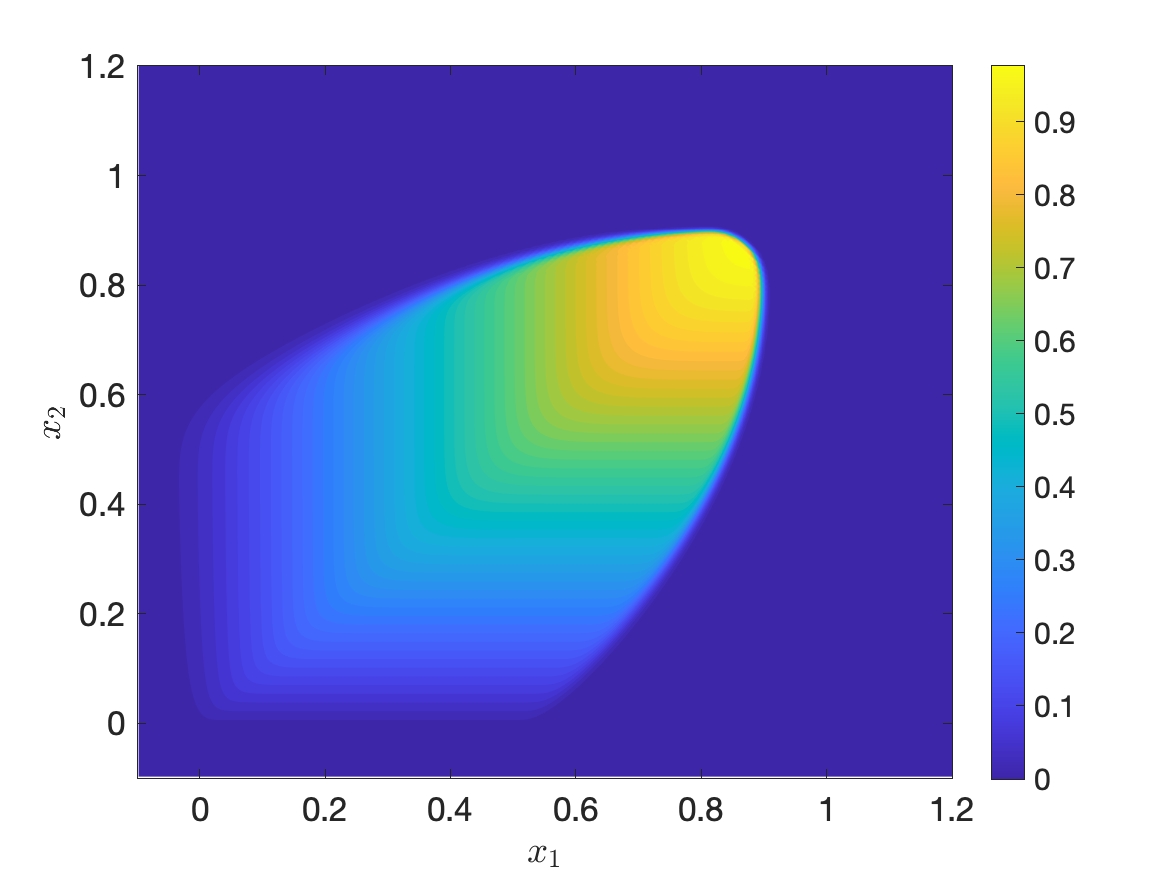}}  
\hfill
\subfloat[Interpolation regime $t=1$: solution overlay]{\includegraphics[width=2.5in]{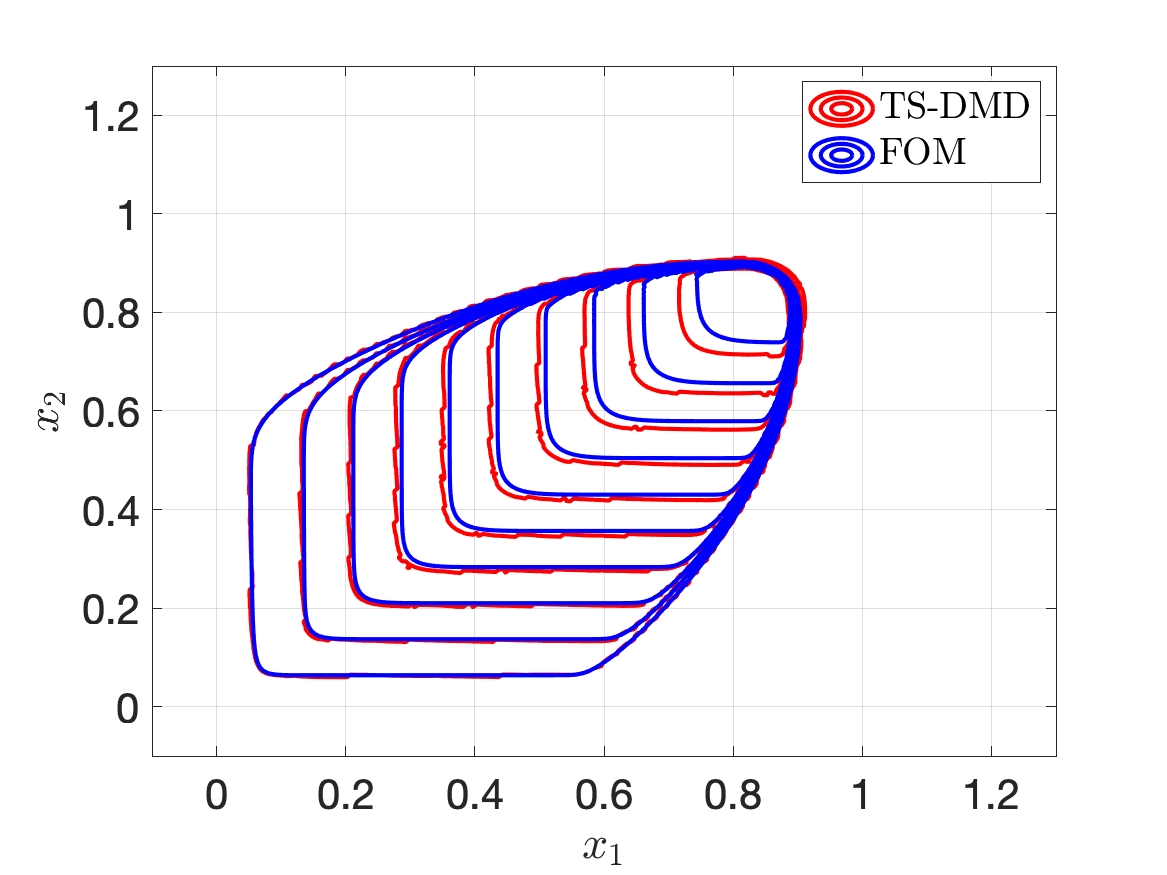}}
\caption{Results for test-3. Solutions for TS--DMD and DMD computed with $n=5$. \label{fig: test-3 sol comp}}	
\end{figure} 

Now consider the extrapolation regime $t\in [1,1.1]$, for which we find the errors 
\begin{equation}
\begin{aligned}
\text{Extrapolation regime} \begin{cases}
\text{TS--DMD}:\hspB & E_a(n=5,1,1.1)=8.0\times 10^{-2},\\
\text{DMD}:\hspB & E_a(n=5,1,1.1)=30.1\times 10^{-2}.
\end{cases}
\end{aligned}
\end{equation}
 Notice that both the methods perform worse than in the interpolation regime. Nevertheless, TS--DMD provides a much better solution with an average that is almost four times smaller than that from DMD. For most practical applications, an error of $30\%$ resulting from DMD is undesirable. The current test case, and the previous ones, indicate that for hyperbolic PDEs, DMD is largely inadequate in the extrapolation regime. Our claim is further corroborated by \Cref{fig: test-3 sol comp extrp} that compares the solutions at $t=1.1$. Compared to the interpolation regime, the DMD solution exhibits stronger oscillations. In contrast, the TS--DMD solution is (almost) oscillation-free and provides an accurate approximation of the HF solution.

Although TS--DMD provides a much better solution than DMD, it has two inadequacies: (i) close to the leading edge, it under predicts the shock location by almost $2\%$, and (ii) it fails to accurately capture the shape of the rarefaction fan---see \Cref{fig: test-3 sol diag}. The linearity of the best-fit evolution operator (i.e., the matrix $\td{\mcal A}$ in \eqref{best-fit}) could be a possible reason behind these inadequacies. Note that the Burgers' equation is non-linear. Consequently, it exhibits a solution-dependent shock location and a rarefaction fan; recall that linear hyperbolic PDEs (with a sufficiently regular velocity field) do not exhibit rarefaction fans and have a solution independent shock-location. The action of the best-fit linear operator may be incapable of accurately predicting these non-linear transport-dominated effects. At least for the Burgers' equation, introducing quadratic nonlinearities via the operator inference (see \cite{Benjamin2016}) approach might improve the results. Further investigation is required to establish anything conclusive.

\begin{figure}[ht!]
\centering
\subfloat[Extrapolation regime $t=1.1$: TS--DMD]{\includegraphics[width=2.5in]{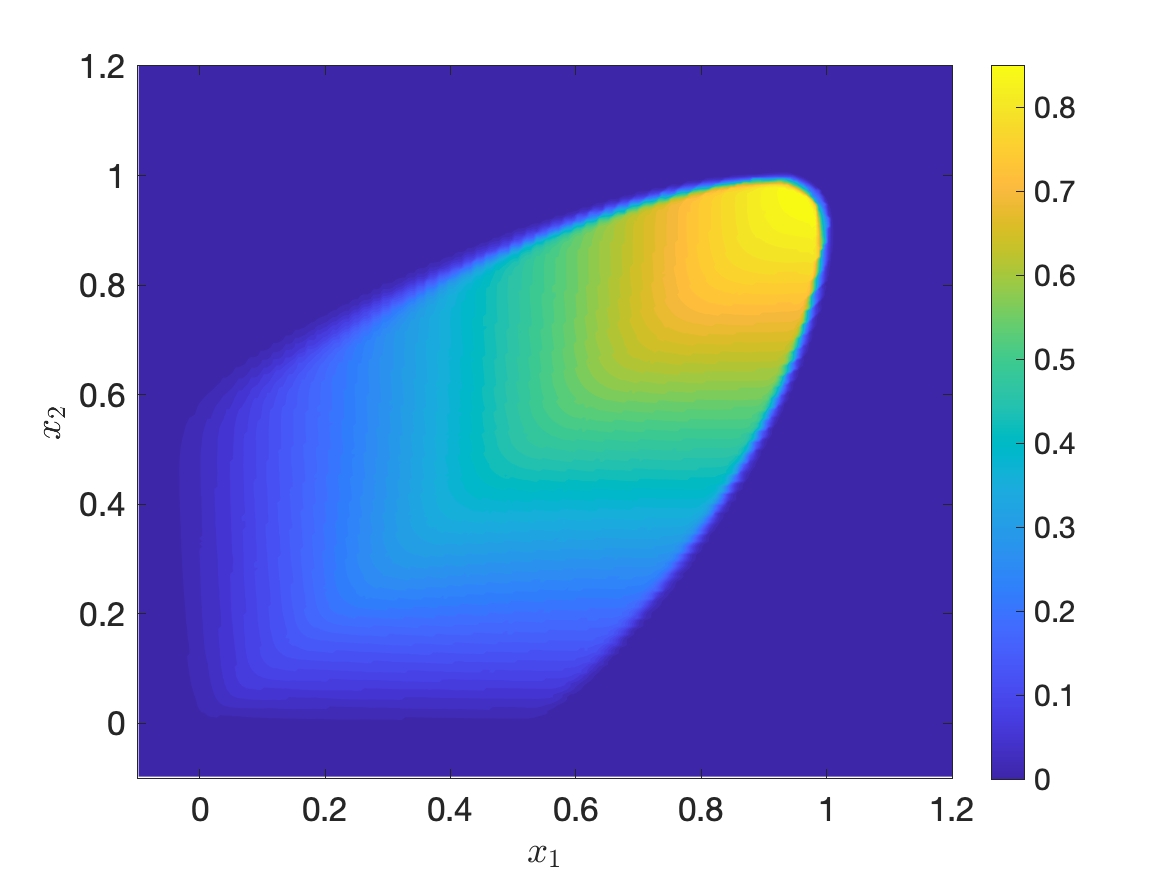}} 
\hfill
\subfloat[Extrapolation regime $t=1.1$: DMD]{\includegraphics[width=2.5in]{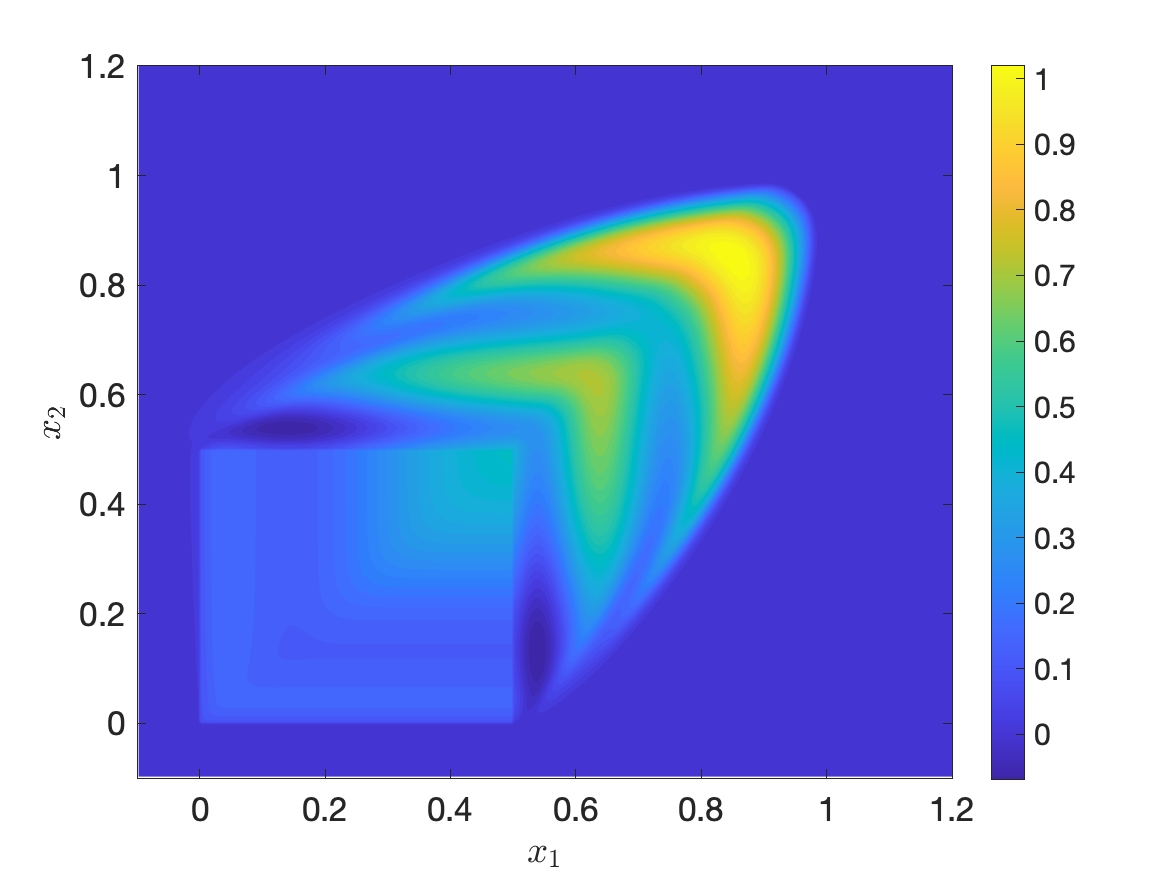}}
\hfill
\subfloat[Extrapolation regime $t=1.1$: HF]{\includegraphics[width=2.5in]{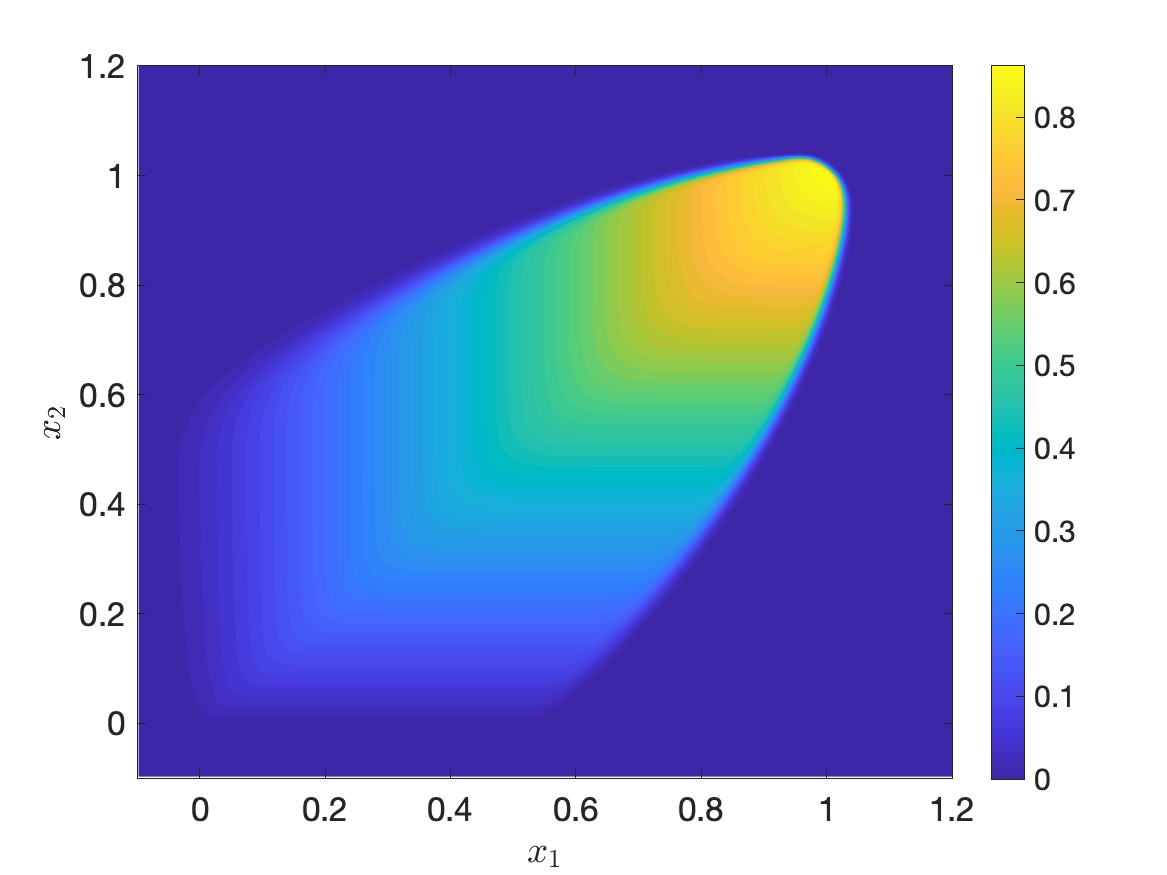}} 
\hfill
\subfloat[Extrapolation regime $t=1.1$: solution overlay]{\includegraphics[width=2.5in]{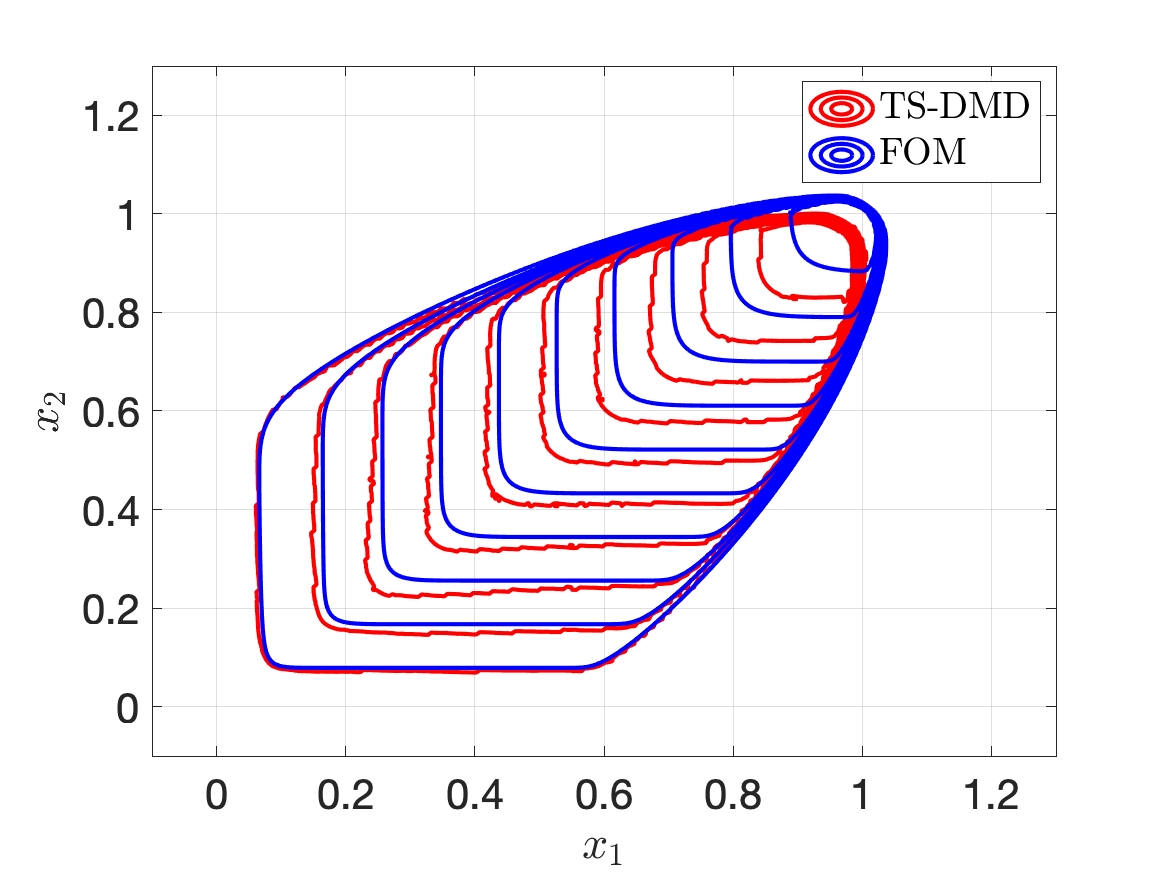}\label{fig: test-3 sol diag}} 
\caption{Results for test-3. Solutions for TS--DMD and DMD computed with $n=5$. \label{fig: test-3 sol comp extrp}}	
\end{figure} 

\Cref{fig: test-3 err time} depicts the temporal behaviour of the error $E(n=5,t)$. TS-DMD provides a much better approximation with error values almost an order of magnitude smaller than in DMD. Similar to the previous test cases, in the extrapolation regime, the error from both the methods increases monotonically with time. Notice that for TS--DMD, the error is particularly large close to $t=0$. This is because at $t=0$, the discontinuity-topology of the solution suddenly changes---see \Cref{sec: snapshots varphi} and Section-6 of \cite{sarna2021datadriven} for further details. As a result, compared to the other time instances, the image registration technique provides a slightly inaccurate spatial transform.

\begin{figure}[ht!]
\centering
\includegraphics[width=2.5in]{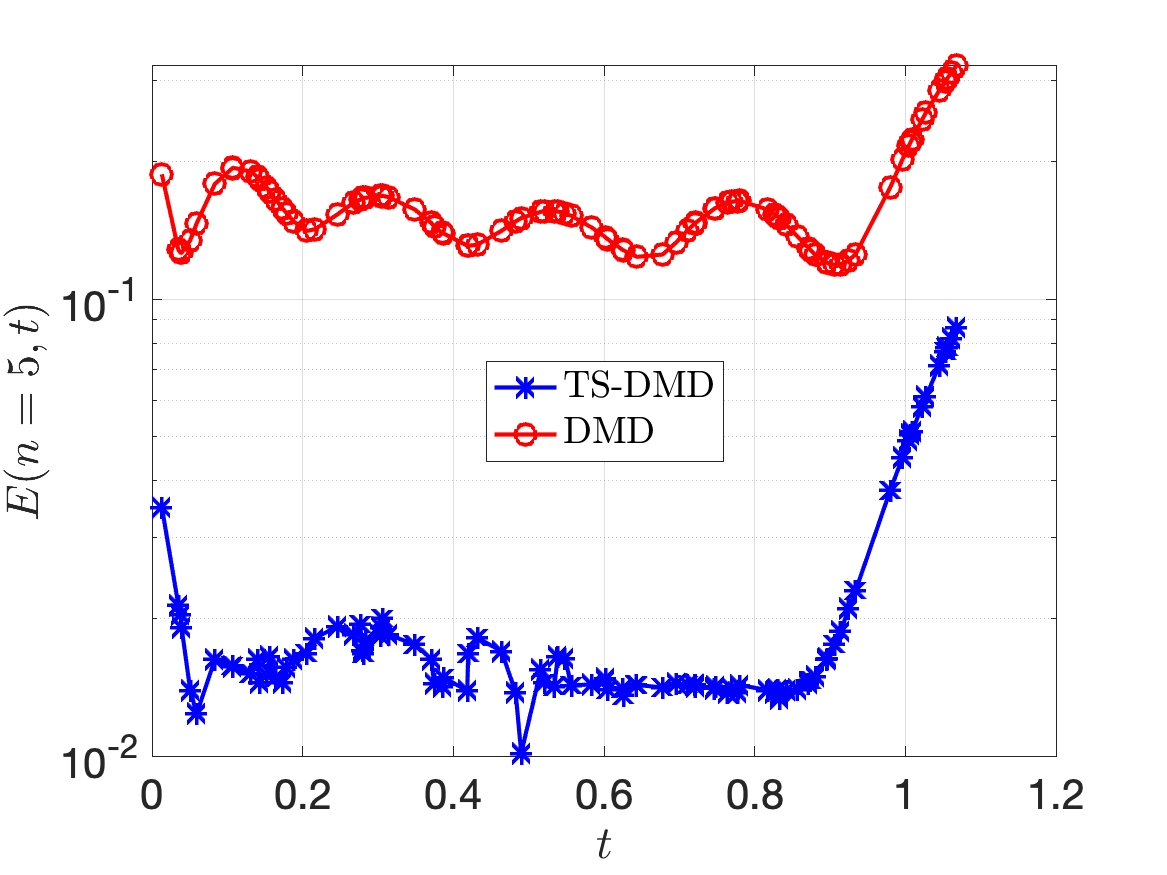}
\caption{Results for test-3. Temporal behaviour of the L1-error defined in \eqref{def E}. Simulations performed with n = 5. \label{fig: test-3 err time}}	
\end{figure} 

\subsubsection{Speed-up vs. accuracy}
Consider the speed-up defined in \eqref{def kappa}. The test set $\parSp_{test}$ consists of $100$ uniformly and independently sampled points from $[0,1.1]$. For $n=5$, TS-DMD results in a speed-up of $1.7\times 10^3$ with an average error of $2.3\%$. 

\section{Conclusions}\label{sec: conclusion}
Solutions to hyperbolic PDEs exhibit temporal discontinuities which hinders their approximability in a linear reduced space. This makes the standard approach of learning a ROM directly for the solution ineffective. We resolve this problem by introducing a de and re-composition step. In the decomposition step, we express the solution as a composition of a transformed solution and a de-transformer. As compared to the solution, both these objects have good approximability in a linear reduced space and allow for efficient learnability via standard ROM learning techniques. We then recover a ROM for the solution via a re-composition.

 Our de and re-composition architecture has mainly three appealing properties. Firstly, it does not require an explicit form of the underlying hyperbolic PDE. Secondly, it is detached from any pre-existing numerical implementation of a ROM learning toolbox. Lastly, de and re-composition can easily be deactivated for problems that are tractable via standard, previously developed, learning techniques. This way, one can avoid the extra cost associated with these two steps.
 
 We used DMD to learn the ROMs and performed numerical experiments involving linear, non-linear, and multi-D (in space) hyperbolic PDEs. Consistent with the previous works, we identified the following drawbacks of a standard DMD approach that learns a ROM directly for the solution. Firstly, both in the interpolation and the extrapolation regime, the ROM is plagued with oscillations. Secondly, in the extrapolation regime, the ROM largely misrepresents the solution and results in an error of fifty to hundred percent. Furthermore, at least empirically, even with a large reduced space, in the extrapolation regime, the ROM does not appear to converge to the true solution. 
 
To a significant extent, our approach overcomes the aforementioned inadequacies. Both inside the interpolation and the extrapolation regime, it results in an error that is at least five to ten times smaller than in the standard approach. Furthermore, owing to the de-transformation step, it results in a ROM that is (almost) oscillation free. However, similar to the standard approach, it results in an error that increases monotonically with time in the extrapolation regime. As a result, accurate extrapolation was only possible for short time intervals.

\section*{Acknowledgements}
\addcontentsline{toc}{section}{Acknowledgments}
N.S  and P.B are supported by the German Federal Ministry for Economic Affairs and Energy (BMWi) in the joint project "MathEnergy - Mathematical Key Technologies for Evolving Energy Grids", sub-project: Model Order Reduction (Grant number: 0324019B).

\appendix
\section{Error bound}\label{app: error bound}
By triangle's inequality, we find 
\begin{equation}
\begin{aligned}
\|\solFV{t}-u_n(\cdot,t)\|_{L^1(\Omega)}\leq &\|g_N(\td\varphi(\cdot,t),t)-g_N(\td{\varphi}_n(\cdot,t),t)\|_{L^1(\Omega)}\\
& + \|g_N(\td{\varphi}_n(\cdot,t),t)-g_n(\td\varphi_n(\cdot,t),t)\|_{L^1(\Omega)}\\
&=: A_1 + A_2.
\end{aligned}
\end{equation}
We first bound $A_1$. Provided our HF solver is total variation diminishing (TVD), we have 
\begin{gather}
\left|\solFV{t}\right|_{BV(\Omega)}\leq \left|\sol{t}\right|_{BV(\Omega)},\label{app: TVD property}
\end{gather}
where $BV(\Omega)$ is a space of functions with bounded variations. Note that the TVD property is an immediate consequence of the CFL condition in Lemma-20.2 of \cite{FVNotes}. Since $\varphi(\cdot,t)$ is a $W^{1,\infty}(\Omega,\mbb R^d)$ homeomorphism, we find 
\begin{gather}
\left|g_N(\cdot,t)\right|_{BV(\Omega)} = \left|u_N(\cdot,t)\right|_{BV(\Omega)}.\label{app: BV with homeo}
\end{gather}
 Then, using the assumptions on $\td\varphi$ and $\td\varphi_n$, Lemma-3.1 of \cite{Welper2017} and the above two relations \eqref{app: TVD property} and \eqref{app: BV with homeo}, we bound $A_1$ via
 \begin{equation}
\begin{aligned}
A_1\leq \left|g_N(\cdot,t)\right|_{BV(\Omega)}\|\td\varphi(\cdot,t)-\td\varphi_n(\cdot,t)\|_{L^{\infty}(\Omega)}
=\left|u_N(\cdot,t)\right|_{BV(\Omega)}\delta \leq \left|u(\cdot,t)\right|_{BV(\Omega)}\delta.
\end{aligned}
\end{equation}
An integral transform provides a bound for $A_2$ that reads
\begin{equation}
\begin{aligned}
A_2 \leq &\|\nabla \td\varphi_n(\cdot,t)\|_{L^{\infty}(\Omega)}\|g_N(\cdot,t)-g_n(\cdot,t)\|_{L^1(\Omega)}\\
\leq &(\|\nabla \td\varphi(\cdot,t)\|_{L^{\infty}(\Omega)} + \delta)\|g_N(\cdot,t)-g_n(\cdot,t)\|_{L^1(\Omega)}. 
\end{aligned}
\end{equation}
\bibliographystyle{abbrv}
\bibliography{../papers}

\end{document}